\documentclass[final,times]{article}
\usepackage{fullpage}
\usepackage{lipsum}
\usepackage{amsfonts}
\usepackage{graphicx}
\usepackage{epstopdf}
\usepackage{tikz}
\usepackage{algorithmic}
\usepackage{booktabs}
\usepackage{amsmath}
\usepackage{amssymb}
\usepackage{amsthm}
\usepackage{mathrsfs}
\usepackage{mathtools}
\usepackage{amsopn}
\usepackage{hyperref}
\usepackage{todonotes}
\usepackage{subcaption}
\usepackage{comment}
\usepackage{authblk}
\graphicspath{{./images/}}

\newcommand{\be}{\begin{equation}}
\newcommand{\ee}{\end{equation}}

\newcommand{\scri}{\mathscr{I}}
\newtheorem{definition}{Definition}[section] 

\newtheorem{remark}{Remark}
\newtheorem{theorem}{Theorem}[section]

  \title{Finite Elements for Helmholtz Scattering with Infinity as a Computational Boundary}

\author[1]{Markus Wess}
\affil[1]{Institute of Analysis and Scientific Computing, Technische Universit\unexpanded{\"a}t Wien, A-1040, Vienna, Austria.\authorcr
    \tt markus.wess@tuwien.ac.at}

\author[2]{An\i l Zengino\u{g}lu}
\affil[2]{Institute for Physical Science and Technology,
            University of Maryland,
            College Park,
            20742,
            MD,
            USA.\authorcr
    \tt anil@umd.edu}

\begin{document}
\maketitle
\begin{abstract}
Building on the null-infinity-layer construction, we develop an \(H^1\)-conforming finite-element formulation of hyperboloidal compactification for the exterior Helmholtz equation. A change of coordinates maps infinity to a finite outer boundary, and a rescaling removes the leading oscillatory decay. We derive the transformed equation and a global sesquilinear weak formulation with bounded coefficients. The compactified boundary contributes an explicit boundary mass term, and its trace gives the far-field pattern up to a known normalization. We compare the resulting method with finite-element discretizations using perfectly matched layers (PML) and report benchmark results in two and three dimensions. Numerical experiments include scattering by a unit disk, resonance in a trapping geometry, a manufactured benchmark in three dimensions, and a submarine benchmark.
\end{abstract}

\tableofcontents

\section{Introduction}
\label{sec:intro}

We study time-harmonic scattering by a bounded obstacle in a homogeneous medium. Let $\mathcal O\subset\mathbb R^d$, with $d\in\{2,3\}$, be a bounded Lipschitz domain with boundary $\partial\mathcal O$, representing the obstacle, and denote by $\mathcal O^c:=\mathbb R^d\setminus\overline{\mathcal O}$ the physical exterior domain. We use $\tilde x\in\mathcal O^c$ for physical coordinates and write $r=\|\tilde x\|$ and $\hat x=\tilde x/r\in\mathbb S^{d-1}$. 
Given a wavenumber $k>0$, the physical radiating scattered field
$U^{\mathrm{s}}:\mathcal O^c\to\mathbb C$ satisfies the exterior Dirichlet scattering problem
\begin{align}
  (\Delta_{\tilde x}+k^2)U^{\mathrm{s}} &= 0
  &&\text{in }\mathcal O^c, \label{eq:helmholtz}\\
  U^{\mathrm{s}} &= -U^{\mathrm{i}}
  &&\text{on }\partial\mathcal O, \label{eq:dirichlet_bc}\\
  \lim_{r\to\infty}
  r^{\frac{d-1}{2}}
  \left(\partial_r U^{\mathrm{s}} - i k U^{\mathrm{s}}\right) &= 0
  &&\text{uniformly in }\hat x .
  \label{eq:sommerfeld}
\end{align}

To solve the Helmholtz equation without artificial truncation, we map the unbounded exterior domain to a bounded computational domain using hyperboloidal compactification \cite{Zenginoglu2011HyperboloidalLayers, Zenginoglu2026NullInfinityLayer}. 
We choose the compactification center to be the origin and assume
$0\in\mathcal O$. The compactified coordinates are denoted by
$x$, with $\rho=\|x\|$, and are related to the physical coordinates by
\begin{equation}\label{eq:compactification}
  \tilde x=\frac{x}{\Omega(\rho)} ,
  \qquad
  r=\|\tilde x\|=\frac{\rho}{\Omega(\rho)}, 
\end{equation}
Outside a radius $R$ with $\overline{\mathcal O}\subset B_R(0)$, the radial compactification maps the unbounded physical radius $r\in[R,\infty)$ to the finite layer $\rho\in[R,S)$, with $\scri:=\{\rho=S\}$ representing infinity. The boundary defining function $\Omega$ is chosen such that $\Omega=1$ at the interface $\rho=R$, $\Omega>0$ for $R\le\rho<S$, and $\Omega(S)=0$. Inside $B_R(0)$ we use the identity map. The compactified computational domain is $\mathcal D:=B_S(0)\setminus\overline{\mathcal O}$, with inner boundary $\partial\mathcal O$ and outer compactified boundary $\scri$. 

The central result of this paper is a weak formulation of the hyperboloidal problem for the homogeneous Helmholtz equation. Hyperboloidal compactification for the Helmholtz equation reformulates the exterior problem on a bounded computational domain by combining radial compactification with a rescaling of the unknown that removes the algebraic decay and leading oscillations of outgoing waves \cite{Zenginoglu2026NullInfinityLayer, zenginouglu2026penrose}. 

\subsection{Literature review}
\label{sec:literature_review}

The exterior Helmholtz problem has two numerical difficulties: the computational domain is unbounded and radiating solutions remain oscillatory at arbitrarily large distances. Both issues have been studied extensively (see, for example, \cite{schot1992eighty,TSYNKOV1998465,HagstromReview,Johnson2021PML,Astley2000IEReview}). A common approach is to truncate the domain and impose an approximate radiation condition at an artificial boundary. Classical local absorbing boundary conditions include those of Engquist--Majda \cite{engquist1977absorbing} and Bayliss--Turkel \cite{bayliss1980radiation}; later work developed practically implementable local conditions of arbitrarily high order \cite{hagstrom1998formulation,givoli2001high}. Such conditions can be interpreted as local approximations to the exterior Dirichlet-to-Neumann map. However, Pad\'e-type local DtN conditions incur a relative error bounded below independently of frequency \cite{galkowski2024local}.

Perfectly matched layers and absorbing layers instead introduce a complex stretch or damping layer around the physical region \cite{Berenger1994PML,Johnson2021PML}. For radial PMLs, the truncation error can decay exponentially with frequency, with the decay rate controlled by the layer width and the complex-scaling angle \cite{galkowski2023perfectly}.
These volumetric truncations are local and straightforward to couple to finite elements, but their accuracy depends on the layer thickness, the stretching or damping profile, and the outer truncation.

By contrast, the exact exterior Dirichlet-to-Neumann map, or an equivalent exact nonreflecting boundary operator, yields an exact truncation on a suitable artificial boundary. Such operators are generally nonlocal and may require dense boundary operators, modal expansions, or, in the time domain, convolution \cite{Givoli1991NRBCReview,GroteKeller1995NRBC, AlpertGreengardHagstrom2000Kernels}. Boundary-integral formulations reformulate constant-coefficient scattering problems as equations on the obstacle boundary and thereby build in the Sommerfeld condition \cite{ColtonKress2013,Kress1991BIEAcoustic}. They are highly effective for homogeneous exterior media and provide far-field quantities naturally, but they rely on the availability of a fundamental solution and become less direct for variable-coefficient or nonlinear problems.

A related method to hyperboloidal compactification is the finite--infinite element method. Infinite elements couple standard finite elements in a bounded near-field region to exterior approximation spaces that encode outgoing asymptotics \cite{Gerdes1998SummaryIE,Astley2000IEReview, DemkowiczShen2006IE}. In conjugated infinite elements, the leading outgoing factor $r^{-(d-1)/2} e^{ikr}$ is incorporated into the exterior basis, and the remaining amplitude is approximated numerically \cite{Gerdes1998ConjugatedIE,DemkowiczShen2006IE}. Hardy space infinite elements and complex-scaled infinite elements include related realizations of the radiation condition through analytic continuation, pole conditions, exterior complex scaling, or radial-PML-related constructions \cite{HohageNannen2009HardyIE,Halla2016HardyIE, NannenSchaedle2011HardyIE,NannenWess2022ComplexIE, HallaKachanovskaWess2025RadialPMLIE}. These methods show that incorporating outgoing asymptotics or complex-scaled decay directly into the exterior discretization can lead to highly accurate solvers. Hyperboloidal compactification is close in spirit to these approaches because it removes a similar leading radial decay and oscillation before discretization. The subleading behavior, however, is freely chosen. This freedom corresponds to a freedom of time transformations in the time domain. Another difference to infinite elements is that the exterior is compactified to a finite layer, so that infinity becomes an actual computational boundary and standard \(H^1\)-conforming finite elements can be used for the rescaled unknown.

Pure spatial compactifications of unbounded domains also have a long history. Grosch and Orszag studied algebraic and exponential maps from \([0,\infty)\) to a finite interval using Chebyshev expansions \cite{GroschOrszag77}; Boyd developed mapped Chebyshev methods for unbounded domains \cite{boyd1982optimization,boyd2001chebyshev}; and later work analyzed convergence and mapping design in higher dimensions \cite{shen2009some,shen2014approximations}. Such mappings are effective for decaying or sufficiently smooth solutions, but they compress the wavelength of oscillatory Helmholtz solutions near the mapped boundary, turning the infinite domain problem into an infinite-resolution problem. This limitation was already noted by Grosch and Orszag, who observed that ``solutions that oscillate out to infinity are not so amenable to these techniques'' \cite{GroschOrszag77}.

A related PML-type method also uses coordinate compression to translate the unbounded exterior problem to a bounded layer. The idea is to combine the radial transformation with a rescaling that extracts the oscillatory factor explicitly so that the layer equation is regular at the boundary \cite{yang2021truly, wang2025novel, wang2026robust}. The method is closely related to the null-infinity layer in that it uses a real radial compression together with an explicit removal of the outgoing oscillatory factors. Its purpose, however, is primarily domain reduction through a real compressed layer. The outer boundary of the layer is a finite artificial boundary, not a compactified representation of radiative fields. Therefore, the far-field pattern is not obtained as the trace of a rescaled unknown at the outer boundary.

Hyperboloidal compactification avoids the wavelength compression by combining spatial compactification with a time transformation adapted to outgoing characteristics, thereby pushing the outer boundary to the radiation zone \cite{Zenginoglu2008ScriFixing,Zenginoglu2011HyperboloidalLayers}. The construction is based on Penrose's conformal treatment of \emph{null infinity} \cite{Penrose1963Asymptotic,Penrose1965ZeroRestMass} and the subsequent studies of wave equations and Einstein equations \cite{gowdy_wave_1981, friedrich1983cauchy, frauendiener2004conformal}.
Outgoing waves never reach spatial infinity given by $r\to\infty$ at standard fixed time. They arrive far away at later and later times, propagating along the characteristics $t-r=$ const. Null infinity, denoted by $\scri$, is the ideal boundary reached by these outgoing characteristics all the way out. It is where radiation ends up, and where the far-field pattern naturally lives. In the frequency domain where the Helmholtz scattering problem is formulated, null infinity is the boundary on which the far-field pattern is defined. Writing
\[
U(r,\hat x) \sim r^{-(d-1)/2} e^{i k r} u_\infty(\hat x) \quad\text{as }r\to\infty,
\]
we see that the far-field reaches null infinity as the leading factor $r^{-(d-1)/2} e^{i k r}$ is removed by the rescaling \cite{zenginouglu2011geometric}. Hyperboloidal compactification includes null infinity in the computational domain thereby avoiding artificial timelike outer boundaries. This program has been especially successful in numerical relativity, including waveform extraction at null infinity, black-hole perturbation theory, and quasinormal-mode computations \cite{zengin2008hyperboloidal, zenginouglu2010asymptotics, ZenginogluKhanna2011NullInfinity, BernuzziNagarZenginoglu2011LargeMassRatio, VanoVinualesHusaHilditch2015,JaramilloPanossoMacedoAlSheikh2021, assaad2025quasinormal, PanossoMacedoZenginoglu2025QNM}.

Despite the extensive use of hyperboloidal compactification in numerical
relativity, it has seen comparatively little use in finite element simulations
of frequency-domain scattering. One reason is technical: most applications in
numerical relativity exploit smooth spherical or nearly spherical geometries,
whereas computational scattering problems often involve nonsmooth obstacles and
non-spherical near-field geometry. Building on \cite{Zenginoglu2026NullInfinityLayer}
the current paper addresses this gap by developing an $H^1$-conforming finite element
formulation of hyperboloidal compactification for the Helmholtz equation. 
The method retains a spherical compactified boundary but allows general obstacle
geometries inside the computational domain.

\subsection{Outline}

Sec.~\ref{sec:hyperboloidal_compactification} reviews the compactification and rescaling and records the transformed coefficients needed below. Sec.~\ref{sec:var} gives the main new analytical component of this work: the global \(H^1\)-conforming variational formulation, including its outer-boundary term and the proof of consistency and boundedness. Sec.~\ref{sec:numerics} demonstrates the resulting finite-element method on 2D and 3D examples and compares it with radial PML discretizations. \ref{sec:ref_sol} gives the analytic disk-scattering reference solution used in the convergence tests.

\section{Hyperboloidal compactification of the Helmholtz equation}\label{sec:hyperboloidal_compactification}

The method of hyperboloidal compactification consists of two main ideas: Compactification of the exterior, unbounded domain and a suitable rescaling of the solution.
First, in Sec.~\ref{sec:compactification}, we analyze what happens when space is compactified in the exterior region. Note that this is purely geometrical and independent of the underlying equation.
While in the beginning we keep all computations general, in Sec.~\ref{sec:radial_compactification}, we restrict ourselves to purely radial compactifications.
Subsequently, in Sec.~\ref{sec:rescaling}, we turn to the Helmholtz equation and apply a suitable rescaling. In Sec.~\ref{sec:compactification_rescaling}, we then apply the compactification to the rescaled equation and analyze the regularity of the transformed equation as well as the relation of the far-field trace to the transformed variable.

\subsection{Compactification}\label{sec:compactification}
We work in $\mathbb{R}^d$ with $d\in\{2,3\}$ equipped with the standard Euclidean inner product and associated norm $\|a\| := \sqrt{a^\top a}$, and $I:=I_d$ denotes the $d\times d$ identity matrix. The compactified coordinate is denoted by $x\in\mathbb R^d$, while the physical coordinate is denoted by $\tilde x\in\mathbb R^d$. Unless stated otherwise, gradients and divergences without a tilde are taken with respect to $x$. The compactified boundary is denoted by $\scri$ (read ``scri'' for script i).
One way to describe the compactified boundary is by a boundary defining
function.

\begin{definition}[Boundary defining function]\label{def:boundary_defining_function}
Let $\mathcal D\subset\mathbb R^d\setminus\overline{\mathcal O}$ be a bounded
Lipschitz compactified computational domain whose boundary decomposes as
$\partial\mathcal D=\partial\mathcal O\,\dot\cup\,\scri$.
A function $\Omega: \overline{\mathcal D} \mapsto \mathbb R_{\ge 0}$ is called a boundary defining function for $\scri$ if
\[
  \scri=\{x\in\overline{\mathcal D}:\Omega(x)=0\},\qquad
  \Omega(x)>0\quad\text{for }x\in\overline{\mathcal D}\setminus\scri,
\]
and
\[
  \nabla\Omega(x)\neq 0,\qquad
  x^\top\nabla\Omega(x)<0\qquad\text{for all }x\in\scri .
\]
\end{definition}

This boundary defining function is essentially Penrose's conformal factor introduced in \cite{Penrose1963Asymptotic}. The definition is independent of the geometry of $\scri$. The compactified
boundary may be any smooth closed hypersurface with spherical topology, provided
the rays from the compactification center meet $\scri$ transversely from the
interior. In later sections, we will restrict the discussion to the radial
setting by choosing $\Omega$ to depend only on $\|x\|$ which implies that the
compactified boundary is circular in two dimensions and spherical in three
dimensions. More general compactified boundaries can have any geometry with
spherical topology.

Given such a function $\Omega$, we consider the compactification
\begin{equation}\label{eq:map}
  \tilde{x}=\frac{x}{\Omega(x)}
\end{equation}
on the region where $\Omega>0$. The following identities are local and hold
where $L:=\Omega-x^\top\nabla\Omega$ is nonzero. The sign condition in
Definition~\ref{def:boundary_defining_function} implies $L>0$ in a neighborhood $\scri$. 
The radial assumptions in
Sec.~\ref{sec:radial_compactification} ensure that this map gives a bijective
global compactification of the exterior radial coordinate.
The Jacobian matrix is
\begin{equation}\label{eq:J}
  J
  = \frac{\partial \tilde{x}}{\partial x}
  = \Omega^{-1}\left(I-\Omega^{-1}\,x\,(\nabla\Omega)^\top\right).
\end{equation}
Whenever $L\neq 0$, we can invert \eqref{eq:J} via the Sherman--Morrison formula:
\begin{equation}\label{eq:Jinv}
  J^{-1}
  = \Omega\left(I+\frac{1}{L}\,x\,(\nabla\Omega)^\top\right).
\end{equation}
The determinant of the Jacobian matrix reads
\begin{equation}\label{eq:detJ}
  \det J = \Omega^{-(d+1)}\,L.
\end{equation}
It follows that on any region where $\Omega>0$ and $L>0$, the compactification map is locally orientation-preserving. For the transformation of second-order operators, we list the coefficient matrices
\begin{equation}\label{eq:Mdef}
  M^{-1}
  := (\det J)\,J^{-1}J^{-\top},
  \qquad
  M := \left(M^{-1}\right)^{-1} = \frac{1}{\det J}\,J^\top J,
\end{equation}
where $J^{-\top}:=(J^{-1})^\top$. A direct substitution of \eqref{eq:Jinv} and \eqref{eq:detJ} into \eqref{eq:Mdef} gives
\begin{align}
  M
  &= \frac{\Omega^{\,d-1}}{L}
     \left[
       I - \frac{1}{\Omega}\left(x\,(\nabla\Omega)^\top + \nabla\Omega\,x^\top\right)
       + \frac{\|x\|^2}{\Omega^2}\,\nabla\Omega\,(\nabla\Omega)^\top
     \right],
  \label{eq:M}\\[4pt]
  M^{-1}
  &= \Omega^{-(d-1)}
     \left[
       L\,I + x\,(\nabla\Omega)^\top + \nabla\Omega \,x^\top
       + \frac{(\nabla\Omega)^\top \nabla\Omega}{L}\,x\,x^\top
     \right].
  \label{eq:Minv}
\end{align}
We remove the singular factor $\Omega^{-(d-1)}$ by defining the rescaled matrix
\begin{equation}\label{eq:rescaledM}
  A := \Omega^{d-1}\,M^{-1}
  = L\,I + x\,(\nabla\Omega)^\top + \nabla\Omega \,x^\top
  + \frac{(\nabla\Omega)^\top \nabla\Omega}{L}\,x\,x^\top.
\end{equation}
This rescaling plays an essential role in the regularity and boundary behavior of the  compactified Helmholtz equation discussed in Sec.~\ref{sec:regularity}.


\subsubsection{Restriction to radial compactification}\label{sec:radial_compactification}

We now specialize Definition~\ref{def:boundary_defining_function} to radial
layer compactifications. Let $0<R<S$, choose the compactification center at the
origin with $0\in\mathcal O$, and assume that
$\overline{\mathcal O}\subset B_R(0)$. We take $\mathcal D=B_S(0)\setminus\overline{\mathcal O}$, $\scri=\partial B_S(0)$, and choose a radial boundary defining function
$\Omega(x)=\Omega(\rho)$, where $\rho:=\|x\|$, such that $\Omega\equiv1$ for
$\rho\le R$, $\Omega>0$ for $\rho<S$, $\Omega(S)=0$, and
$\Omega'(S)<0$. The compactification map sends the finite layer $R\le \rho<S$ to the unbounded physical exterior $R\le r<\infty$ by
\[
  r=\frac{\rho}{\Omega(\rho)},
\]
with the boundary at infinity represented by $\rho=S$.

Since the compactification center lies inside the obstacle, $\rho>0$ throughout
the computational domain. A prime denotes differentiation with respect to the
compactified radius $\rho$. Define the unit radial vector
$n:=x/\rho\in\mathbb S^{d-1}$ and $L(\rho):=\Omega(\rho)-\rho\,\Omega'(\rho)$.
On the layer $\nabla\Omega=\Omega' n$. Moreover,
\[
  \frac{dr}{d\rho}
  =\frac{\Omega-\rho\Omega'}{\Omega^2}
  =\frac{L}{\Omega^2}.
\]
The boundary defining function gives $L(S)=-S\Omega'(S)>0$; we choose the
radial compactification so that $L$ is strictly positive on the whole layer.
Therefore, the radial map $\rho\mapsto r$ is strictly increasing and, since the
angular variable is unchanged, $x\mapsto\tilde x$ is bijective between the
compactified layer and the physical exterior. We define
\[
  G:=\frac{d\rho}{dr}=\frac{\Omega^2}{L},
\]
with continuous extension $G(S)=0$. Substituting into
\eqref{eq:rescaledM} yields
\begin{equation}\label{eqn:A_matrix}
  A
  = L\,I+\left(2\rho\,\Omega'
      +\frac{\rho^2(\Omega')^2}{L}\right)n n^\top
  = L\,I+(G-L)n n^\top .
\end{equation}
Here we used $\Omega=L+\rho\Omega'$. We decompose $\mathbb R^d$ into the radial
direction and its orthogonal complement using the orthogonal projectors
\begin{equation}\label{eqn:projection}
  Q:=n n^\top,\qquad P:=I-Q.
\end{equation}
These satisfy the usual orthogonal projector identities,
$P^2=P$, $Q^2=Q$, $PQ=QP=0$, and $P+Q=I$. We can then write the transformation
matrix as
\begin{equation}\label{eq:Aproj}
  A=L\,(I-n n^\top)+G\,n n^\top=L\,P+G\,Q .
\end{equation}
The matrix $A$ acts as the scaling $L$ on tangential components and as the
scaling $G$ on the radial component:
\begin{equation}\label{eq:An_property}
  A n=G\,n,
  \qquad
  A v=L\,v \ \ \text{for all } v\in\mathbb R^d
  \text{ with } n^\top v=0 .
\end{equation}
Thus, $A$ has eigenvalue $L$ with multiplicity $d-1$ in the tangential
directions and eigenvalue $G$ with multiplicity $1$ in the radial direction.
The radial conormal degenerates at the compactified boundary: since the outward
unit normal on $\scri$ is $\nu=n$, we have $A\nu=G n$ and $G(S)=0$. This degeneracy is the mechanism by which the compactified weak formulation
avoids an essential boundary condition at infinity.

\subsection{The rescaled Helmholtz equation}\label{sec:rescaling}
Let $\tilde{x}\in\mathbb{R}^d$ denote the physical, noncompactified Cartesian coordinate and consider the homogeneous Helmholtz equation
\begin{equation}\label{eq:helmholtz_tilde}
  \widetilde{\Delta}U + k^2 U = 0,
\end{equation}
where $k>0$ is the wavenumber and $\widetilde{\nabla}$, $\widetilde{\Delta}$ denote the gradient and Laplacian with respect to $\tilde{x}$.

Let $w$ be a smooth nonvanishing complex weight and write $U=w u$. The product rule yields
\begin{equation*}
  \widetilde{\Delta}(w u)
   = \widetilde{\nabla}\cdot \widetilde{\nabla}(w u)
   = \widetilde{\nabla}\cdot\left(u\,\widetilde{\nabla}w + w\,\widetilde{\nabla}u\right)
  = u\,\widetilde{\Delta}w + 2\,\widetilde{\nabla}w\cdot \widetilde{\nabla}u + w\,\widetilde{\Delta}u.
\end{equation*}
Dividing \eqref{eq:helmholtz_tilde} by $w$ gives the equation for $u$:
\begin{equation}\label{eq:u_prelog}
  \widetilde{\Delta}u + 2\,\frac{\widetilde{\nabla}w}{w}\cdot \widetilde{\nabla}u
  + \frac{\widetilde{\Delta}w}{w}\,u + k^2 u  = 0.
\end{equation}
Using the identity
$\widetilde{\Delta}w/w=\widetilde{\Delta}(\ln w)+
\widetilde{\nabla}(\ln w)\cdot\widetilde{\nabla}(\ln w)$,
\begin{equation}\label{eqn:rescaled}
  \widetilde{\Delta}u
  + 2\,\widetilde{\nabla}(\ln w) \cdot \widetilde{\nabla}u
  + \left(\widetilde{\Delta}(\ln w) + \widetilde{\nabla}(\ln w) \cdot\widetilde{\nabla}(\ln w) + k^2\right)u
  = 0.
\end{equation}

\subsection{Compactification of the rescaled Helmholtz equation}\label{sec:compactification_rescaling}
We now apply the compactification introduced in Sec.~\ref{sec:compactification} to the rescaled Helmholtz equation \eqref{eqn:rescaled}. For a scalar field $\tilde{f}(\tilde{x})$ we denote its pullback by $f(x) := \tilde{f}(\tilde{x}(x))$. We use the standard transformation rules
\begin{equation}\label{eq:trafo_rules}
  \widetilde{\nabla}\tilde{f} = J^{-\top}\nabla f,
  \qquad
  \det J \,\widetilde{\nabla}\cdot \tilde{v}
  = \nabla\cdot\left(\det J \,J^{-1} v\right),
  \qquad
  \det J\,\widetilde{\Delta}\tilde{f}
  = \nabla\cdot\left(M^{-1}\nabla f\right),
\end{equation}
where all quantities on the right-hand side are evaluated at $x$ and $v(x):=\tilde{v}(\tilde x(x))$.

Multiplying \eqref{eqn:rescaled} by $\det J$ and applying the transformation rules \eqref{eq:trafo_rules} gives
\begin{equation}\label{eqn:g_trafo}
  \nabla\cdot\left(M^{-1}\nabla u\right)
  + 2\,(\nabla\ln w)^\top M^{-1}\nabla u
  + \left[
      \nabla\cdot\left(M^{-1}\nabla\ln w\right)
      + (\nabla\ln w)^\top M^{-1}\nabla\ln w
      + (\det J)\,k^2
    \right]u = 0.
\end{equation}
To obtain an explicitly regular expression at infinity, we multiply the equation by $\Omega^{d-1}$ and use $A := \Omega^{d-1}\,M^{-1}$, which extends regularly to the compactification boundary as discussed in Sec.~\ref{sec:compactification}. Using the identity
\begin{equation}\label{eq:div_identity}
  \Omega^{d-1}\,\nabla\cdot\left(M^{-1}X\right)
  = \nabla\cdot\left(A X\right) - (d-1)\,(\nabla\ln\Omega)^\top A X,
\end{equation}
with $X=\nabla u$ and $X=\nabla\ln w$, and the determinant formula $\det J=\Omega^{-(d+1)}L$ from \eqref{eq:detJ}, we write the rescaled equation as
\begin{equation}\label{eqn:transformed_helmholtz}
  \nabla\cdot\left(A\,\nabla u\right)
  + \beta^\top\nabla u + C u = 0,
\end{equation}
where
\begin{equation}\label{eqn:BC}
  \beta = A\nabla\ln\frac{w^2}{\Omega^{d-1}},
  \qquad
  C = \nabla\cdot\left(A\nabla\ln w\right) +
  \left(\nabla\ln\frac{w}{\Omega^{d-1}}\right)^\top
  A\nabla\ln w + \frac{k^2L}{\Omega^2}.
\end{equation}
The last term in $C$ is generally singular at $\scri=\{\Omega=0\}$ unless the phase of the weight $w$ is chosen to cancel it to leading order, which we discuss next.

\subsubsection{Regularity at infinity}\label{sec:regularity}
Under a suitable choice of weight $w$, we can show that the resulting transformed equation has desirable properties at the compactified boundary. In this section, we develop some intuition for the choice of the weight and present the coefficients for a few examples.  

The weight $w$ can be interpreted as a combination of time transformation and conformal rescaling \cite{Zenginoglu2011HyperboloidalLayers, Zenginoglu2026NullInfinityLayer, zenginouglu2026penrose}. It is chosen so that the transformed coefficients extend regularly to the boundary $\scri$ and the unknown has nonoscillatory asymptotics. For brevity of the discussion, we focus on the radial case in the rest of this paper. We have
\begin{equation}\label{eq:grad_r}
  \nabla \rho = n,
  \qquad
  \nabla r = \frac{dr}{d\rho}\,\nabla\rho = \frac{1}{G}\,n,
  \qquad
  \nabla\cdot\left(\psi(\rho)\,n\right) = \psi'(\rho) + \frac{d-1}{\rho}\,\psi(\rho),
\end{equation}
for any radial scalar function $\psi(\rho)$.

For the choice of the weight, we note that the far-field behavior of homogeneous solutions decomposes into incoming and outgoing spherical waves in the exterior domain. We use the time dependence $e^{-ikt}$. With this convention, $e^{+ikr}$ represents an outgoing radial phase and $e^{-ikr}$ an incoming radial phase:
\begin{equation}\label{eq:decomposition}
U(r,\hat x) \sim r^{-\frac{d-1}{2}}\left( e^{-ikr}\,u_\infty^-(\hat x)+e^{+ikr}\,u_\infty^+(\hat x)\right),
\qquad \text{as } \ r\to\infty.
\end{equation}
Here $\hat x\in \mathbb S^{d-1}$ denotes the observation direction, $u_\infty^-(\hat x)$ is the incoming radiation, and $u_\infty^+(\hat x)$ is the far-field of the outgoing scattered field. The decomposition above suggests the weight
\begin{equation}\label{eq:characteristic}
  w(r) = r^{-\frac{d-1}{2}}e^{\pm i k r}.
\end{equation}
The amplitude factor $r^{-(d-1)/2}$ describes the spherical far-field decay. In compactified variables, $r^{-(d-1)/2}=(\Omega(\rho)/\rho)^{(d-1)/2}$, which corresponds to the choice $f(\rho)=1/\rho$ in the more general ansatz \eqref{eq:weight}. Since the compactification center is placed inside the excluded obstacle, $\rho$ is bounded away from zero on the computational domain, and this factor is regular there. Since $w$ depends on $x$ only through $r$, we may differentiate $\ln w$ with respect to $r$ and use \eqref{eq:grad_r}:
\[
  \nabla\ln w
  = \frac{d}{dr}\ln w \,\nabla r
  = \left(\pm i k - \frac{d-1}{2r}\right)\frac{1}{G}\,n
  = \left(\pm i k - \frac{(d-1)\Omega(\rho)}{2\rho}\right)\frac{1}{G}\,n.
\]
Substituting this expression into \eqref{eqn:transformed_helmholtz} and exploiting the radial form $A = L\,P + G\,Q$ from \eqref{eq:Aproj} yields the reduced equation
\begin{equation}\label{eq:char_eq}
  \nabla\cdot\left(A\,\nabla u\right)
  + \left(\pm 2 i k - \frac{(d-1)G}{\rho}\right)\,n^\top\nabla u
  + \frac{(d-1)(3-d)}{4\rho^2}\,L\,u
  = 0.
\end{equation}
The equation extends regularly to the compactified boundary. The $k^2$-contribution in \eqref{eqn:transformed_helmholtz} cancels under the choice \eqref{eq:characteristic}, leaving only a \emph{linear} dependence on $k$ in the first-order transport term. The choice \eqref{eq:characteristic} is the radial phase-amplitude factor underlying classical conjugated infinite elements. The hyperboloidal layer used below retains the same outgoing phase behavior at infinity but allows more flexible choices in the interior of the layer. Geometrically, this choice corresponds to a characteristic time transformation combined with a conformal rescaling of the field.

We can make more general choices of the time transformation. The
weight need not agree with the outgoing characteristic factor
\eqref{eq:characteristic} throughout the layer; it is sufficient that
this agreement holds asymptotically at the compactified boundary. We
therefore employ the ansatz
\begin{equation}\label{eq:weight}
  w(\rho)
  =
  \left[f(\rho)\Omega(\rho)\right]^{\frac{d-1}{2}}
  e^{i k h(r(\rho))},
\end{equation}
where $h=h(r)$ is a real-valued radial height function of the
physical radius $r(\rho)$. The derivative of $h$, pulled back to the compactified layer, is the corresponding boost function.

\begin{definition}[Height and boost functions]\label{def:boost_function}
Let $\Omega$ be a radial boundary defining function as in
Sec.~\ref{sec:radial_compactification}, and let
\[
  r(\rho)=\frac{\rho}{\Omega(\rho)}, \qquad 0\le \rho<S .
\]
Let $[R,S]$ denote the compactified layer. A height function for the layer is a real-valued radial function $h=h(r)$ with $h(r)=0$ for $0\le r\le R$. Its associated boost function is defined by
\[
  H(\rho)
  :=
  \left.\frac{dh}{dr}\right|_{r=r(\rho)},
  \qquad R<\rho<S .
\]
We require that $H$ extends to the compactified boundary and satisfies
\[
  H(S)=1,
  \qquad
  \frac{1-H^2}{G}\in L^\infty(R,S),
\]
where $H(S)$ denotes the boundary value of this extension and $G$ is
the compactification factor associated with $\Omega$.
\end{definition}

The terms height and boost come from the hyperboloidal
literature in numerical relativity and refer to time transformations \cite{Zenginoglu2011HyperboloidalLayers}.
The condition $H(S)=1$ makes the layer asymptotically
characteristic. The boundedness of $(1-H^2)/G$ is the regularity
condition needed at the compactified boundary. In the finite-element
formulation below, we use height functions that are continuous and
piecewise smooth. Since $h\equiv0$ and $\Omega\equiv 1$ in the inner region $0\le r\le R$, the transformation is confined to the layer.

All primes below denote differentiation with respect to the compactified radius $\rho$. Using \eqref{eq:grad_r} and $\nabla\ln\Omega = (\ln\Omega)'\,n$ on the layer, we obtain
\begin{equation}\label{eq:grad_ln_w}
  \nabla\ln w
  = i k\,H(\rho)\,\nabla r + \frac{d-1}{2}\,\nabla(\ln f \Omega)
  = \left( i k\,\frac{H}{G} + \frac{d-1}{2}\,(\ln f \Omega)'\right)n.
\end{equation}
With this choice, the first-order term in \eqref{eqn:transformed_helmholtz} simplifies substantially because
\[
  \nabla\ln\left(\frac{w^2}{\Omega^{d-1}}\right)=(d-1)\nabla\ln f+2ik\nabla h,
\]
and $A\nabla h=Hn$ (using $An=Gn$ and $\nabla h=(H/G)n$). Denoting $\sigma:=(\ln f)'$, we get 
\[ 
  \beta=A\nabla\ln(w^2/\Omega^{d-1})=(2ikH+(d-1)G\sigma)n=:Bn,
\] 
and
\begin{align}\label{eq:C}
  C
  &=
  k^2\,\frac{1-H^2}{G}
  + i k\left(H' + \frac{d-1}{\rho}H\right)
  + \frac{d-1}{4\rho^2}\left[(3-d)(L-G) + 2\rho\, G'\right] \nonumber \\
  & + i k (d-1) \sigma H + \frac{(d-1)}{2} \left[ (\sigma G)' + \frac{(d-1)\sigma G}{2\rho} (2+\rho \sigma) \right].
\end{align}
For $f=1$, we have $\sigma=0$, $\beta=2ikHn$, and
\[
  C
  =
  k^2\frac{1-H^2}{G}
  +ik\left(H'+\frac{d-1}{\rho}H\right)
  +R_\Omega,
  \qquad
  R_\Omega
  =
  \frac{d-1}{4\rho^2}
  \left[(3-d)(L-G)+2\rho G'\right].
\]
For $f=1/\rho$, we have $\sigma=-1/\rho$, and the coefficient simplifies to
\[
  \beta=\left(2ikH-\frac{(d-1)G}{\rho}\right)n,
  \qquad
  C
  =
  k^2\frac{1-H^2}{G}
  +ikH'
  +\frac{(d-1)(3-d)}{4\rho^2}L .
\]
Thus the exact spherical amplitude $f=1/\rho$ gives the reduced equation \eqref{eq:char_eq} when $H=\pm1$. For the finite-element weak formulation later, we use the simpler choice $f=1$.

Among the terms above, the only potentially singular contribution at the compactification boundary $\scri=\{\Omega=0\}$ is $k^2\,\frac{1-H^2}{G}$. The boundedness requirement in Definition~\ref{def:boost_function} is equivalently the local condition
\begin{equation}\label{eq:H_condition}
  1-H^2=\mathcal O(G)
  \qquad\text{as }\rho\to S.
\end{equation}
For a boundary defining function with a simple zero at $S$, we have $G(\rho)=\mathcal O((S-\rho)^2)$. Therefore \eqref{eq:H_condition} is stronger than the pointwise outgoing condition $H(S)=1$. Since $H$ is smooth on the layer, a sufficient local condition is $H(S)=1$ and $H'(S)=0$. In the numerical constructions below we impose the stronger asymptotic condition $H=1+\mathcal O(G)$, for example by taking $H=1-G$ in the compactification layer.

\subsubsection{The far-field pattern}
Among the main advantages of compactification is that the far-field pattern can be extracted from the global solution without post-processing or extrapolation. When $H=1+\mathcal O(G)$ near $\scri$, the height function satisfies $h(r)=r+h_0+o(1)$ for some constant $h_0$ as $r\to\infty$. If the outgoing physical field has far-field pattern $U_\infty(\hat x)$, in the sense that
\[
  U(r,\hat x)
  =
  r^{-\frac{d-1}{2}}e^{ikr}
  \left(U_\infty(\hat x)+o(1)\right),
  \qquad r\to\infty,
\]
then the compactified unknown satisfies (for the choice $f=1$)
\[
  u|_{\scri}(\hat x)
  =
  S^{-(d-1)/2}e^{-ikh_0}U_\infty(\hat x).
\]
The far-field pattern is obtained from the trace at $\scri$ by an explicit constant factor. This formula uses the far-field convention specified by the asymptotic expansion above; alternative scattering conventions, for example those including dimension-dependent constants or phase shifts, require the corresponding constant conversion.

\begin{remark}[Radiation condition and regularity at $\scri$]
  Under the assumption $H=1+\mathcal O(G)$ near $\scri$, an outgoing term $r^{-(d-1)/2}e^{ikr}U_\infty^+(\hat x)$ is transformed into a function with a finite trace at $\scri$. By contrast, an incoming term $r^{-(d-1)/2}e^{-ikr}U_\infty^-(\hat x)$ produces a rapidly oscillatory factor proportional to $e^{-ik(r(\rho)+h(r(\rho)))}$, whose $\rho$-derivative grows like $1/G$ near $\scri$. Since $G=O((S-\rho)^2)$ for a simple zero of $\Omega$, this incoming contribution is not in $H^1(\mathcal D)$ unless $U_\infty^-=0$. Therefore, the compactified $H^1(\mathcal D)$ trial space excludes the incoming branch and selects the outgoing branch.
\end{remark}

\section{Variational formulation}\label{sec:var}
The strong form of the transformed equation is given in \eqref{eqn:transformed_helmholtz} with the definitions \eqref{eq:Aproj} and \eqref{eqn:BC}. We now provide an independent derivation of the variational formulation for the transformed Helmholtz equation used in the finite element discretization. 
The strong form is used on the smooth compactification layer, while the weak form below is the global statement on $\mathcal D$ with the identity transformation in the noncompactified interior.

The formulation is obtained from the physical weak form before collecting all lower-order terms into the strong-form coefficient $C$. This is the natural form for computation because it keeps the amplitude rescaling in a symmetric first-order form and leaves only bounded coefficients at the compactified boundary.

From this point on we work using the radial compactification of
Sec.~\ref{sec:radial_compactification}. We use the corresponding radial
boundary defining function and an admissible radial height function with boost as in Definition~\ref{def:boost_function}.
We state the theorem for homogeneous Dirichlet data on the obstacle, so that the test functions vanish on $\partial\mathcal O$. With the interior extensions $h=0$ and $\Omega=1$ for $\rho\le R$, the compactified unknown agrees with the physical unknown near the obstacle. Neumann or impedance data may therefore be imposed through the standard physical boundary terms on $\partial\mathcal O$.

\begin{theorem}[Consistency and boundedness]\label{thm:compactified_weak_form}
  Let $\mathcal O,\mathcal D$ be as above, $\Omega,L,G$ be a radial compactification as in Sec.~\ref{sec:radial_compactification} and $H,h$ be as in Definition \ref{def:boost_function}. Moreover let $U$ be the solution to the homogeneous Helmholtz equation \eqref{eq:helmholtz} with Dirichlet data $U_0$ at $\partial \mathcal O$ and Sommerfeld radiation condition such that the pull back $u$ of the rescaled physical solution given by 
$$
  U(\tilde x(x)) =
  \Omega(\rho)^{\frac{d-1}{2}} e^{ik h(r(\rho))}u(x),
$$
  is in $H^1(\mathcal D)$. Then $u$ is a solution to the weak formulation to find $u\in\mathcal V_1 := \{v\in H^1(\mathcal D): \mathrm{tr}|_{\partial\mathcal O}v=U_0\}$ such that
$$
  a(u,v)=0 \qquad \forall v\in\mathcal V_0:= \{v\in H^1(\mathcal D): \mathrm{tr}|_{\partial\mathcal O}v=0\},
$$

%
where
\begin{align}
  \begin{split}
  a(u,v) :={}&
  -\int_{\mathcal D} \nabla u\cdot (GQ+LP)\nabla \bar v
  + k^2\int_{\mathcal D} \frac{1-H^2}{G}\,u\bar v
  -\int_{\mathcal D} \frac{1}{L}
    \left(\frac{d-1}{2}\Omega'\right)^2 u\bar v \\
  &-ik \int_{\mathcal D}
    H\,n\cdot\left(u\nabla \bar v-\bar v \nabla u\right)
   -\int_{\mathcal D}
    \frac{\Omega}{L}\left(\frac{d-1}{2}\Omega'\right)
    n\cdot\left( u \nabla \bar v+\bar v \nabla u\right)
   +ik\int_{\scri} u\bar v .
\end{split}
\label{eq:compactified_weak_radial_full}
\end{align}
  All coefficients in \eqref{eq:compactified_weak_radial_full} are bounded on $\mathcal D$. Moreover, the sesquilinear form $a$ is bounded on $\mathcal H^1(\mathcal D)\times\mathcal H^1(\mathcal D)$: there exists a constant $C_a=C_a(k,\Omega,H,\mathcal D)>0$ such that
\[
  |a(u,v)|
  \le
  C_a
  \|u\|_{H^1(\mathcal D)}
  \|v\|_{H^1(\mathcal D)}
  \qquad
  \forall u,v\in\mathcal H^1(\mathcal D).
\]
\end{theorem}
\begin{remark}
No essential boundary condition is imposed at $\scri$; the only outer-boundary contribution is the phase term
\[
  ik\int_{\scri}u\bar v .
\]
\end{remark}
\begin{remark}
  Since for $\rho<R$ we have $G=L=1,H=0,\Omega'=0$ we immediately obtain that for $v$ with compact support in $B_R(0)\setminus\mathcal O$ the weak formulation is 
  \begin{align*}
    -\int_{B_R(0)\setminus\mathcal O}\nabla u\cdot \nabla v+k^2\int_{B_R(0)\setminus\mathcal O}uv = 0,
  \end{align*}
  i.e., the volume terms of the classical weak form of the Helmholtz equation.
\end{remark}

\begin{proof}
We first derive the formula for functions that are smooth on the noncompactified interior and on the compactification layer, and then extend the resulting sesquilinear form to $H^1(\mathcal D)$ by density. Set $\alpha:=(d-1)/2$. For $\varepsilon>0$, let $\mathcal D_\varepsilon:=\{x\in\mathcal D:\rho<S-\varepsilon\}$, $\widetilde{\mathcal D}_\varepsilon:=\tilde x(\mathcal D_\varepsilon)$,
  and denote the artificial outer boundary by $\Sigma_\varepsilon:=\{x\in\mathcal D:\rho=S-\varepsilon\}$. On the truncated physical domain $\widetilde{\mathcal D}_\varepsilon$, testing with a smooth function $V$ with $\mathrm{tr}|_{\partial O}=0$, integration and integration by parts for $(\widetilde\Delta+k^2)U=0$ gives
\[
  -\int_{\widetilde{\mathcal D}_\varepsilon}
       \widetilde\nabla U\cdot\widetilde\nabla\overline V
  +k^2\int_{\widetilde{\mathcal D}_\varepsilon}U\overline V
  +\int_{\partial\widetilde{\mathcal D}_\varepsilon}
       \partial_{\widetilde\nu}U\,\overline V
  =0 .
\]
Insert
\[
  U=\Omega^{\alpha}e^{ikh}u,
  \qquad
  V=\Omega^{\alpha}e^{ikh}v,
\]
and pull the integrals back by $\tilde x=x/\Omega(\rho)$. Using \eqref{eq:detJ} and $G=\Omega^2/L$, we have $\det J=1/\left(G\Omega^{d-1}\right)$. For convenience we define
\[
  p:=\nabla\ln\Omega^\alpha
  =
  \frac{\alpha\Omega'}{\Omega}n,
  \qquad
  q:=\nabla h
  =
  \frac{H}{G}n,
  \qquad
  \mu:=\frac{\Omega}{L}\alpha\Omega' .
\]
Then
\begin{equation}\label{eq:properties_pqmu}
  Ap=\mu n,
  \qquad
  Aq=Hn,
  \qquad
  p\cdot Ap=\frac{1}{L}(\alpha\Omega')^2,
  \qquad
  q\cdot Aq=\frac{H^2}{G}.
\end{equation}
The volume part of the pulled-back weak form is
\[
\begin{aligned}
  I_\varepsilon
  :={}&
  -\int_{\mathcal D_\varepsilon}
  \left(\nabla u+up+ik\,uq\right)
  \cdot A
  \left(\nabla\bar v+\bar vp-ik\,\bar vq\right)
  +k^2\int_{\mathcal D_\varepsilon}\frac{1}{G}u\bar v .
\end{aligned}
\]
Expanding the product gives
\[
\begin{aligned}
  I_\varepsilon
  ={}&
  -\int_{\mathcal D_\varepsilon}\nabla u\cdot A\nabla\bar v
  -\int_{\mathcal D_\varepsilon}
    \left(
      u\,p\cdot A\nabla\bar v
      +
      \bar v\,\nabla u\cdot Ap
    \right) \\
  &-\int_{\mathcal D_\varepsilon}u\bar v\,p\cdot Ap
  +ik\int_{\mathcal D_\varepsilon}
      \bar v\,\nabla u\cdot Aq
  -ik\int_{\mathcal D_\varepsilon}
      u\,q\cdot A\nabla\bar v \\
  &+ik\int_{\mathcal D_\varepsilon}
      u\bar v\left(p\cdot Aq-q\cdot Ap\right)
  -k^2\int_{\mathcal D_\varepsilon}
      u\bar v\,q\cdot Aq
  +k^2\int_{\mathcal D_\varepsilon}\frac{1}{G}u\bar v .
\end{aligned}
\]
The mixed amplitude-phase term vanishes because $p\cdot Aq=q\cdot Ap$. Using \eqref{eq:Aproj} and \eqref{eq:properties_pqmu}, we obtain
\[
\begin{aligned}
  I_\varepsilon
  ={}&
  k^2\int_{\mathcal D_\varepsilon}
    \frac{1-H^2}{G}\,u\bar v
  -\int_{\mathcal D_\varepsilon}
    \nabla u\cdot (GQ+LP)\nabla\bar v
  -\int_{\mathcal D_\varepsilon}
    \frac{1}{L}(\alpha\Omega')^2u\bar v \\
  &-ik\int_{\mathcal D_\varepsilon}
    H\,n\cdot\left(u\nabla\bar v-\bar v\nabla u\right)
  -\int_{\mathcal D_\varepsilon}
    \frac{\Omega}{L}\alpha\Omega'\,
    n\cdot\left(u\nabla\bar v+\bar v\nabla u\right).
\end{aligned}
\]
This is the volume contribution in
\eqref{eq:compactified_weak_radial_full}, restricted to
$\mathcal D_\varepsilon$.

We now compute the boundary contribution. The physical conormal term pulls back to
\[
  B_\varepsilon :=
  \int_{\partial\mathcal D_\varepsilon}
  \bar v
  \left[ A\nabla u +uAp +ik\,uAq
  \right]\cdot\nu 
  =
  \int_{\partial\mathcal D_\varepsilon}
    \bar v\,\nabla u\cdot A\nu
  +\int_{\partial\mathcal D_\varepsilon}
    \mu\,n\cdot\nu\,u\bar v
  +ik\int_{\partial\mathcal D_\varepsilon}
    H\,n\cdot\nu\,u\bar v ,
\]
where $\nu$ is the outward unit normal to $\mathcal D_\varepsilon$. On the obstacle boundary $\partial\mathcal O$, the test function satisfies $v|_\partial\mathcal O=0$, so this contribution vanishes in the homogeneous Dirichlet case. On $\Sigma_\varepsilon$, we have $\nu=n$, and hence $A\nu=An=Gn$. Therefore
\[
  B_\varepsilon
  =
  \int_{\Sigma_\varepsilon}
    G\,\partial_\rho u\,\bar v
  +\int_{\Sigma_\varepsilon}
    \frac{\Omega}{L}\alpha\Omega'\,u\bar v
  +ik\int_{\Sigma_\varepsilon}
    H\,u\bar v ,
\]
where $\partial_\rho u=n\cdot\nabla u$. The boundary term simplifies considerably since $G(S)=0=\Omega(S)$, $H(S)=1$ and $L>0$. We have
\[
  B_\varepsilon
  \longrightarrow
  ik\int_{\scri}u\bar v
  \qquad
  \text{as }\varepsilon\downarrow0 .
\]
Letting $\varepsilon\downarrow0$ gives
\eqref{eq:compactified_weak_radial_full} for functions that are smooth on the interior and on the layer.

It remains to verify boundedness.
By the radial compactification assumptions, $\Omega$, $\Omega'$, $L$, and $G$
are bounded on the layer $[R,S]$, and the lower bound on $L$ gives
$L^{-1}\in L^\infty(R,S)$.
In the noncompactified interior these coefficients are constant. The compactification center lies inside the obstacle, hence $\rho\ge\rho_0>0$ on $\mathcal D$. All coefficients in the remaining terms are bounded on $\mathcal D$; the only
coefficient requiring a separate assumption is $(1-H^2)/G$, which is bounded by Definition~\ref{def:boost_function}. Therefore, Cauchy--Schwarz gives
\[
\begin{aligned}
  |a(u,v)|
  \le{}&
  C_1\|u\|_{L^2(\mathcal D)}\|v\|_{L^2(\mathcal D)}
  +C_2\|\nabla u\|_{L^2(\mathcal D)}
       \|\nabla v\|_{L^2(\mathcal D)} \\
  &+C_3
  \left(
    \|u\|_{L^2(\mathcal D)}\|\nabla v\|_{L^2(\mathcal D)}
    +
    \|v\|_{L^2(\mathcal D)}\|\nabla u\|_{L^2(\mathcal D)}
  \right) 
  +|k|\,\|u\|_{L^2(\scri)}\|v\|_{L^2(\scri)} ,
\end{aligned}
\]
where the constants $C_j$ depend only on $k$, $\Omega$, $H$, and
$\mathcal D$. Since $\mathcal D$ is a bounded Lipschitz domain, the trace
theorem gives
\[
  \|w\|_{L^2(\scri)}
  \le
  C_{\mathrm{tr}}\|w\|_{H^1(\mathcal D)}
  \qquad
  \forall w\in H^1(\mathcal D).
\]
Therefore there exists $C_a=C_a(k,\Omega,H,\mathcal D)>0$ such that
\[
  |a(u,v)|
  \le
  C_a
  \|u\|_{H^1(\mathcal D)}
  \|v\|_{H^1(\mathcal D)}
  \qquad
  \forall u,v\in \mathcal H^1(\mathcal D).
\]
The formula was derived for smooth functions. By the density of smooth functions satisfying the homogeneous trace condition on $\partial\mathcal O$ in $\mathcal V_0$, and by the boundedness just proved, the sesquilinear form extends uniquely and continuously to $\mathcal H^1(\mathcal D)\times\mathcal H^1(\mathcal D)$. Thus the compactified weak formulation is regular up to $\scri$, and the only outer-boundary contribution is the phase term
\[
  ik\int_{\scri}u\bar v .
\]
\end{proof}

\section{Numerical experiments}\label{sec:numerics}
We present a series of numerical experiments to demonstrate the applicability of our method.
The coefficient functions
$L$, $G$, $\Omega'$, and $H$ are evaluated pointwise on the compactified mesh.
No Dirichlet, Neumann, or absorbing boundary condition is imposed at $\scri$;
the compactified boundary contributes only the boundary mass term
$ik\int_{\scri}u\bar v$. The far-field pattern is then recovered from the trace
of $u$ on $\scri$ using the formula in Sec.~\ref{sec:compactification_rescaling}.

Except for the manufactured benchmark in Sec.~\ref{sec:sphere}, the experiments use the standard sound-soft scattering convention. For the disk, trapping geometry, and submarine, we compute the outgoing scattered field \(U^{\mathrm{s}}\) with boundary data \(U^{\mathrm{s}}=-U^{\mathrm{i}}\) on \(\partial\mathcal O\). In Sec.~\ref{sec:sphere}, by contrast, we prescribe Dirichlet data from a known outgoing solution so that the discretization error can be measured directly. The compactified scattered unknown is denoted by $u^{\mathrm{s}}$. The physical total field is $U^{\mathrm{tot}}=U^{\mathrm{i}}+U^{\mathrm{s}}$, and its compactified counterpart is denoted by $u^{\mathrm{tot}}$.

For all hyperboloidal compactification runs, we choose $R$ such that $B_R(0)\supset\mathcal O$ and define the layer coefficients
\[
  \Omega(\rho):=\frac{S-\rho}{S-R},
  \qquad
  H(\rho):=1-G(\rho),
  \qquad R\le \rho\le S .
\]
On the noncompactified interior $\rho<R$, we use the extensions $\Omega=1$ and $H=0$. This gives a characteristic-preserving setting.
The compactified domain is thus naturally decomposed into the disjoint sets $\mathcal D_{int}=B_R(0)\setminus\mathcal O$ and the compactification layer $\mathcal D_{ext}=B_S(0)\setminus B_R(0)$.

We generate simplex meshes with appropriate curved geometry approximation and a given mesh size, fitting the interface $\rho=R$ between the noncompactified interior and the layer.
We use high-order $H^1$-conforming finite element basis functions to discretize the compactified weak formulation \eqref{eq:compactified_weak_radial_full}.
All experiments were implemented in the high-order finite element library Netgen/NGSolve \cite{netgen,ngsolve}.

For comparison with perfectly matched layers, we choose a radial PML based on the complex coordinate transformation
\begin{align*}
  \tilde x \mapsto
  \tilde x+\frac{i\sigma_{\mathrm{PML}}}{k}(\|\tilde x\|-R)\frac{\tilde x}{\|\tilde x\|},
\end{align*}
where $\sigma_{\mathrm{PML}}>0$ is the real damping strength. An outgoing wave $e^{ikr}$ is transformed into $e^{ikr}e^{-\sigma_{\mathrm{PML}}(r-R)}$ in the PML layer. We truncate the resulting equation at $\|\tilde x\|=S$, so the truncation error is expected to be of roughly order $e^{-\sigma_{\mathrm{PML}}(S-R)}$.

\subsection{Scattering by a disk}
In the first experiment, we choose a simple geometry to demonstrate convergence of the method. The geometry breaks all rotational symmetry to avoid possible superconvergence while still allowing a semi-analytic reference solution. We also use this setting to compare the convergence of our method with standard PMLs.
We choose the obstacle $\mathcal O:=B_{R_{\mathrm{scat}}}(-0.2,-0.2)$ with $R_{\mathrm{scat}}=1$ and impose the physical incident plane wave $U^{\mathrm{i}}=\exp(ik\tilde x_1)$ for some wavenumber $k>0$ at the boundary of the obstacle $\partial \mathcal O$. The interface radius is $R=1.5$, and mapped infinity is chosen as $S=2$.
A typical mesh and approximate compactified scattered solution $u^{\mathrm{s}}$ are shown in Figure~\ref{fig:circle_sol_hyp}. The corresponding far-field pattern is shown in Figure~\ref{fig:circle_farfield}.
We refer to \ref{sec:ref_sol} for the construction of reference solutions for these problems.

\subsubsection{Convergence of the hyperboloidal compactification}
To study convergence of the method, we apply successive mesh refinement and measure the relative $L^2$ error against the analytic physical scattered field $U^{\mathrm{s}}$ pulled back to $\mathcal D_{int}$ as well as the relative $L^2(\scri)$ error of the far field. Figure~\ref{fig:circle_conv} shows the convergence of the respective errors for $k=10$ with respect to the mesh size. The measured slopes agree with the expected $\mathsf{h}^{p+1}$ behavior for polynomial order $p$.
\begin{figure}[h!]
  \centering
  \begin{subfigure}[t]{0.45\textwidth}
    \includegraphics[width=\textwidth]{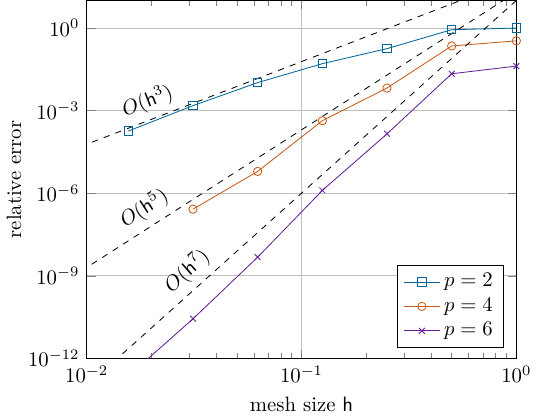}
    \caption{Relative $L^2$ error in $\mathcal D_{int}$ for various polynomial orders}
     \label{fig:circle_conv_inner}
  \end{subfigure}
  \begin{subfigure}[t]{0.45\textwidth}
    \includegraphics[width=\textwidth]{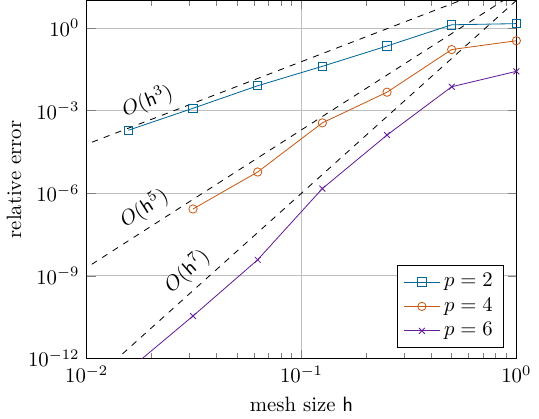}
    \caption{Relative $L^2$ error of the far field for various polynomial orders $p$}
    \label{fig:circle_conv_ff}
  \end{subfigure}
  \caption{Convergence of the hyperboloidal layer for the disk benchmark at $k=10$ for various polynomial orders $p$.}
  \label{fig:circle_conv}
\end{figure}

\subsubsection{Comparison to perfectly matched layers}
To compare the hyperboloidal compactification layer with perfectly matched layers, we use identical configurations of $\mathcal D_{int}$ and $\mathcal D_{ext}$ to define a linear radial perfectly matched layer in $\mathcal D_{ext}$ (cf. \cite{Berenger1994PML,NannenWess2018,Johnson2021PML}). We use identical finite element discretizations for both methods, resulting in complex system matrices with identical dimensions and sparsity patterns and thus similar computational effort.

Figure~\ref{fig:circle_sol_pml} shows a typical PML solution. Note that in $\mathcal D_{ext}$ the PML solution decreases exponentially with the radius and therefore does not lead to a meaningful far-field pattern.

\begin{figure}[h!]
  \centering
  \begin{subfigure}[t]{0.45\textwidth}
    \includegraphics[width=\textwidth]{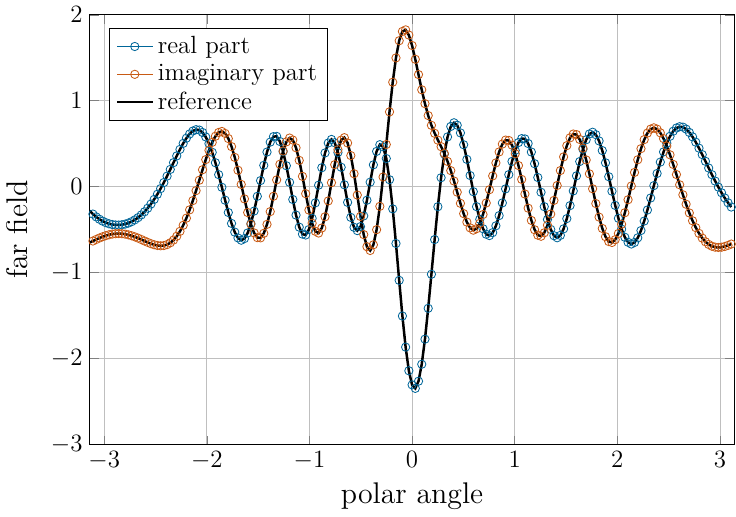}
    \caption{Real and imaginary part of the far field.}
     \label{fig:circle_farfield_val}
  \end{subfigure}
  \begin{subfigure}[t]{0.45\textwidth}
    \includegraphics[width=\textwidth]{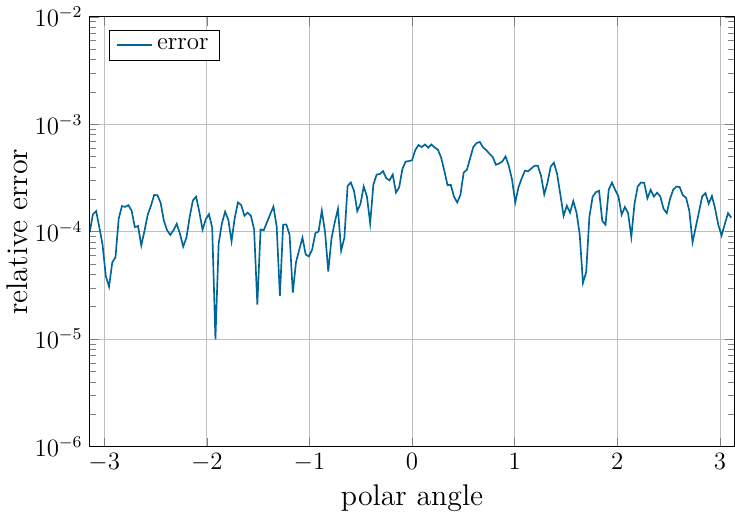}
    \caption{Pointwise error of the far field.}
    \label{fig:circle_farfield_error}
  \end{subfigure}
  \caption{Real part, imaginary part and error of the numerical far field using hyperboloidal compactification for mesh size $0.1$, polynomial order $3$, and wavenumber $k=10$.}
  \label{fig:circle_farfield}
\end{figure}

\begin{figure}[h!]
  \centering
  \begin{subfigure}[t]{0.45\textwidth}
    \includegraphics[width=\textwidth,clip, trim={220 100 20 100}]{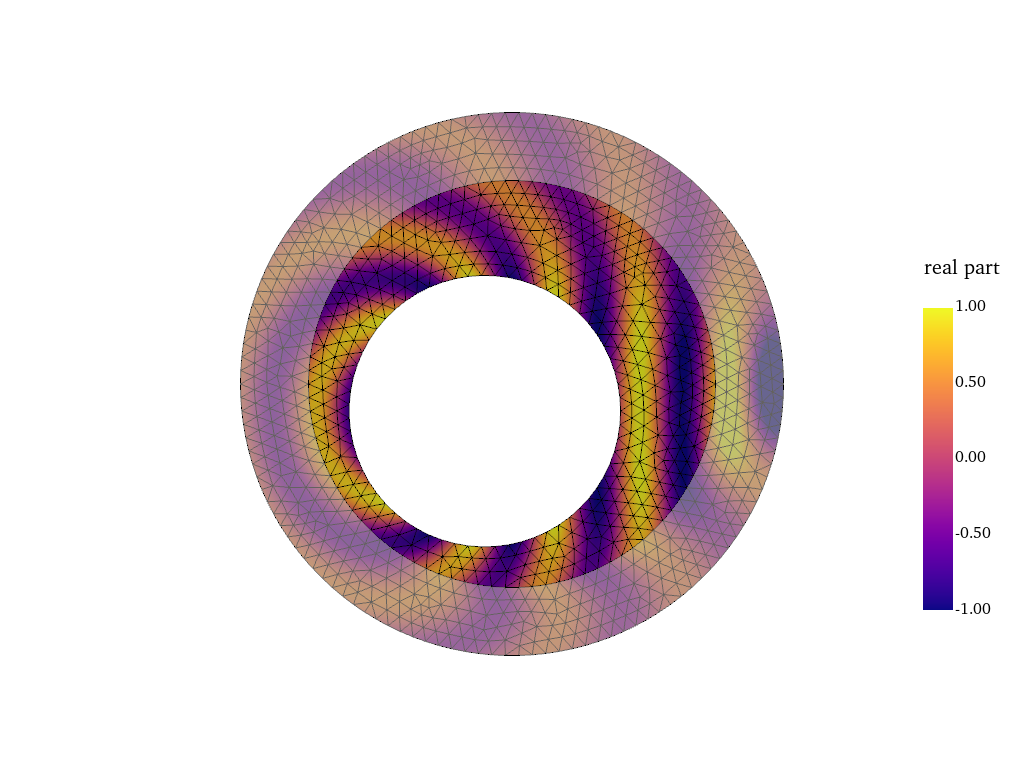}
    \caption{Hyperboloidal compactification.}
     \label{fig:circle_sol_hyp}
  \end{subfigure}
  \begin{subfigure}[t]{0.45\textwidth}
    \includegraphics[width=\textwidth,clip, trim={220 100 20 100}]{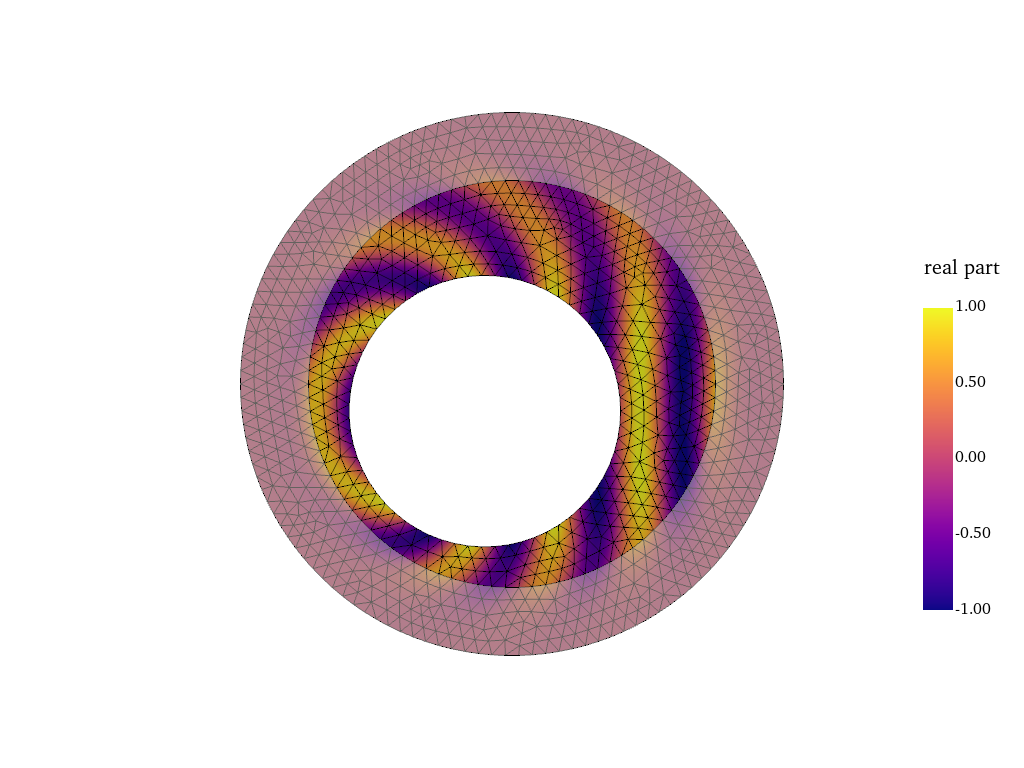}
    \caption{PML with $\sigma_{\mathrm{PML}}=10$.}
    \label{fig:circle_sol_pml}
  \end{subfigure}
  \caption{Real part of the numerical scattered field in the computational domain, computed using hyperboloidal compactification and PML, respectively, for mesh size $0.1$, polynomial order $3$, and wavenumber $k=10$.}
  \label{fig:circle_sol}
\end{figure}

Figure~\ref{fig:circle_conv_pml_order} shows the convergence of the relative $L^2(\mathcal D_{int})$ error of the PML simulation. As expected from theory, the measured slopes again agree with the $\mathsf h^{p+1}$ behavior for polynomial order $p$. For $p=6$, the error saturates near $10^{-8}$ for the displayed PML damping strength, indicating that the discretization error has fallen below the PML truncation error. This error floor can be reduced using a higher damping strength. Figure~\ref{fig:circle_conv_pml_params} shows that for larger PML damping strengths the truncation error is no longer visible in the displayed range, although the error constant also deteriorates.

\begin{figure}[h!]
  \centering
  \begin{subfigure}[t]{0.45\textwidth}
    \includegraphics[width=\textwidth]{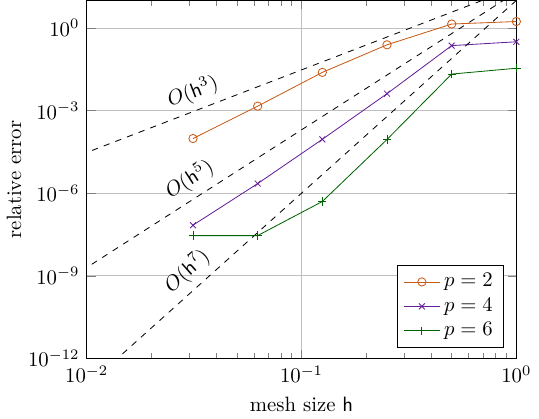}
    \caption{Relative $L^2$ error of the PML solution in $\mathcal D_{int}$ for various polynomial orders $p$ and fixed radial PML damping strength $\sigma_{\mathrm{PML}}=20$.}
     \label{fig:circle_conv_pml_order}
  \end{subfigure}
  \begin{subfigure}[t]{0.45\textwidth}
    \includegraphics[width=\textwidth]{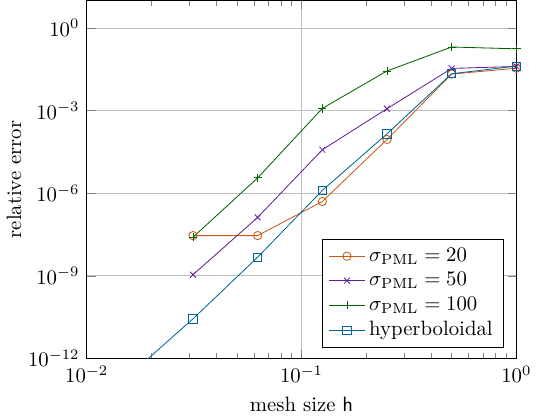}
    \caption{Relative $L^2$ error of the PML solution for different radial PML damping strengths $\sigma_{\mathrm{PML}}$ and fixed polynomial order $p=6$.}
     \label{fig:circle_conv_pml_params}
  \end{subfigure}
  \caption{Comparison of hyperboloidal layers with PML methods.}
     \label{fig:circle_conv_pml}
\end{figure}
\subsubsection{Dependence on the wavenumber}
To test the robustness of our method with respect to the wavenumber, we fix
polynomial order $6$ and mesh size $1/8$. Figure~\ref{fig:circle_wavenumbers} shows the relative $L^2(\mathcal D_{int})$ error of the hyperboloidal compactification and the PML with fixed real damping strength $\sigma_{\mathrm{PML}}=20$ for varying wavenumbers. The wide scan uses geometrically distributed wavenumbers from $10^{-5}$ to $10^2$, while the right panel samples the transition region linearly with $k=5,5.5,\ldots,20$. We observe the pollution effect, i.e., the expected deterioration of the error with increasing wavenumber for both methods. For small wavenumbers the hyperboloidal compactification substantially outperforms the PML. Around $k\approx 10$, however, the PML error is slightly smaller for this fixed discretization and damping parameter.

\begin{figure}
  \centering
  \begin{subfigure}[t]{0.45\textwidth}
    \includegraphics[width=\textwidth]{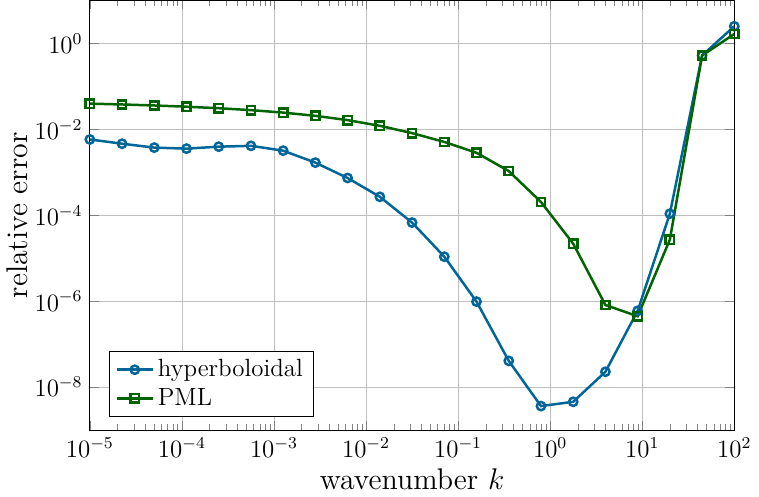}
    \caption{Geometrically distributed wavenumbers from $10^{-5}$ to $10^2$.}
    \label{fig:circle_wavenumbers_log}
  \end{subfigure}
  \begin{subfigure}[t]{0.45\textwidth}
    \includegraphics[width=\textwidth]{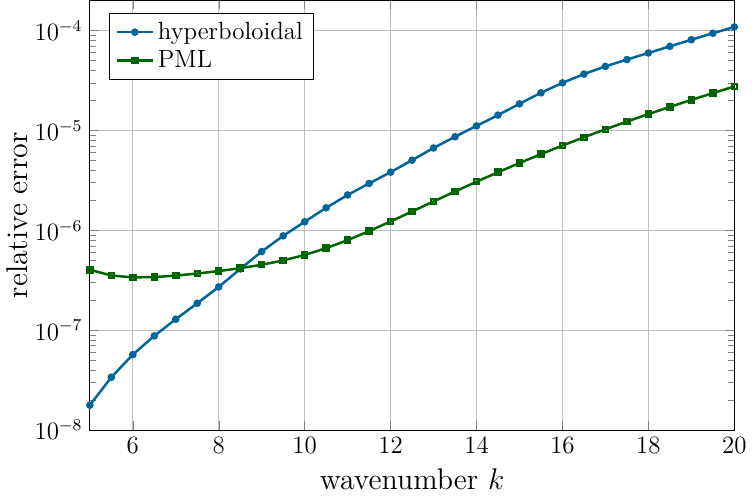}
    \caption{Linearly distributed wavenumbers in the transition region around $k=10$.}
    \label{fig:circle_wavenumbers_zoom}
  \end{subfigure}
    \caption{Relative $L^2(\mathcal D_{int})$ error of the hyperboloidal and PML solution with $\sigma_{\mathrm{PML}}=20$ for polynomial order $6$, mesh size $1/8$, and varying wavenumber $k$.}
    \label{fig:circle_wavenumbers}
\end{figure}

\subsection{Trapping geometry}
As a second example, we choose a trapping geometry (cf. Figure~\ref{fig:trapping_square_mesh}) and again compare the hyperboloidal compactification with PML methods.

\begin{figure}
  \centering
    \includegraphics[width=0.5\textwidth]{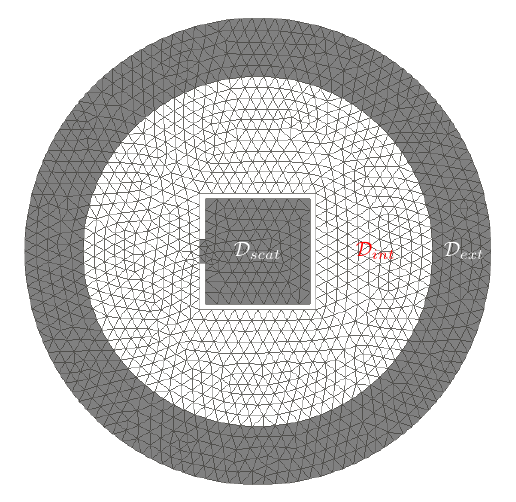}
    \caption{Geometry and mesh of the trapping square experiments.}
     \label{fig:trapping_square_mesh}
\end{figure}

The obstacle is a square cavity with exterior side length $R_{\mathrm{scat}}=1$, wall thickness $D_{\mathrm{scat}}=0.05$, and an aperture of height $D_{\mathrm{hole}}=0.2$ on the left side. We use $S=2$ for the outer compactified boundary, interface radius $R=1.5$, mesh size $0.1$, polynomial order $4$, and incident direction $d=(\cos\theta_{\mathrm{inc}},\sin\theta_{\mathrm{inc}})$ with $\theta_{\mathrm{inc}}=\pi/4$. The wavenumber scan uses $k=1,1.1,\ldots,19.9$, and the PML comparison uses the real damping strength $\sigma_{\mathrm{PML}}=2k$.
As the wavenumber varies, we expect resonant behavior at a discrete set of wavenumbers.
Figure~\ref{fig:trapping_square_wavenumbers} shows the scattering coefficient
\begin{align*}
  C_{scat}:=\frac{\|u^{\mathrm{tot}}\|_{L^2(\mathcal D_{scat})}}{\|u^{\mathrm{tot}}\|_{L^2(\mathcal D_{int})}},
\end{align*}
where $\mathcal D_{scat}$ is the interior of the trapping square (cf. Figure~\ref{fig:trapping_square_mesh}). We observe the expected peaks at certain wavenumbers as well as close agreement between the solutions obtained by PML and hyperboloidal compactification.
\begin{figure}
  \centering
    \includegraphics[width=0.8\textwidth]{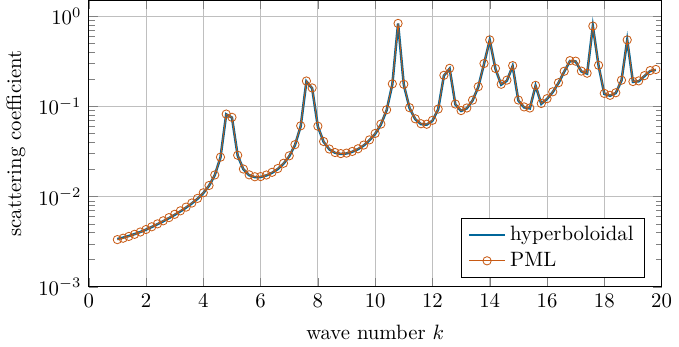}
    \caption{Scattering coefficient for the trapping square problem obtained with hyperboloidal compactification and PML for varying wavenumbers.}
     \label{fig:trapping_square_wavenumbers}
\end{figure}

Figures~\ref{fig:trapping_square_scattered} and \ref{fig:trapping_square_total} show the scattered and total fields for two wavenumbers, one non-resonant and one resonant respectively.
\begin{figure}[h!]
  \centering
  \begin{subfigure}[t]{0.45\textwidth}
    \includegraphics[width=\textwidth,clip, trim={220 100 20 100}]{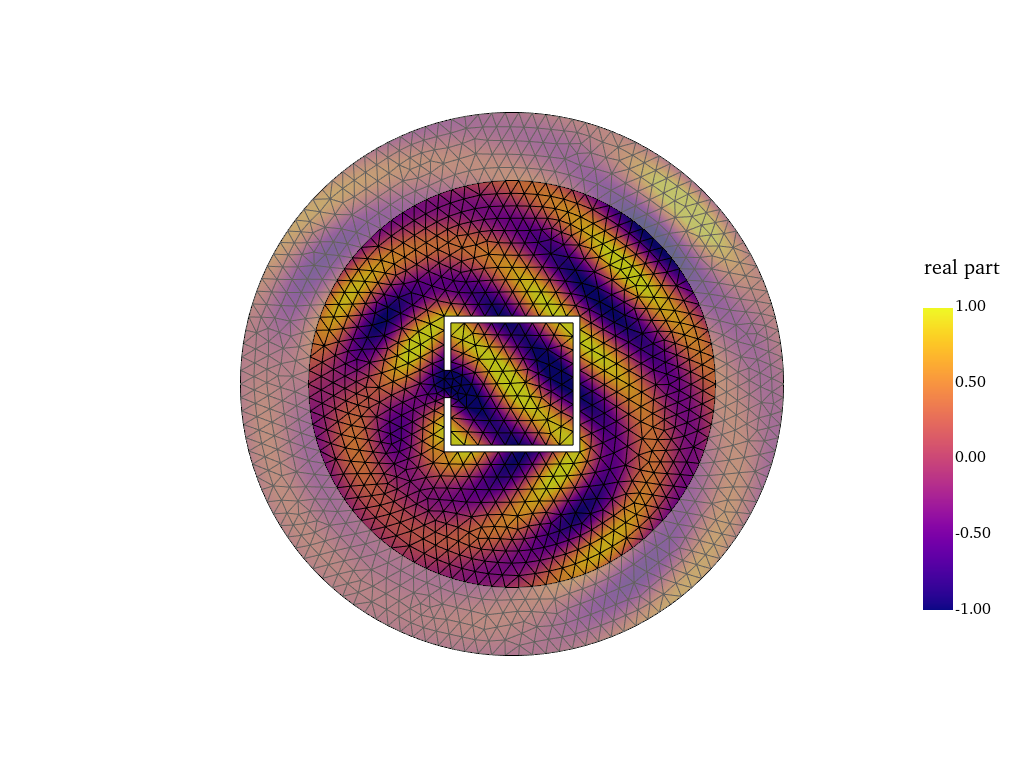}
    \caption{$k=10.6$.}
     \label{fig:trapping_square_scattered_nonres}
  \end{subfigure}
  \begin{subfigure}[t]{0.45\textwidth}
    \includegraphics[width=\textwidth,clip, trim={220 100 20 100}]{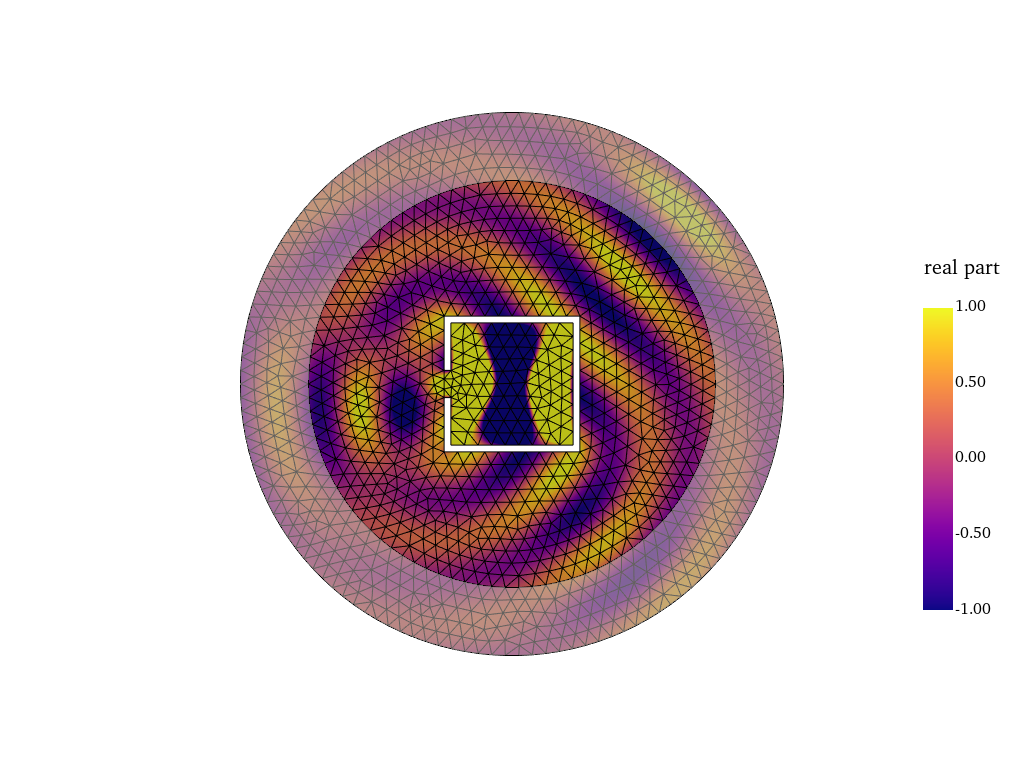}
    \caption{$k=10.8$.}
    \label{fig:trapping_square_scattered_resonant}
  \end{subfigure}
  \caption{Scattered fields for a non-resonant and a resonant wavenumber.}
  \label{fig:trapping_square_scattered}
\end{figure}
\begin{figure}[h!]
  \centering
  \begin{subfigure}[t]{0.45\textwidth}
    \includegraphics[width=\textwidth,clip, trim={220 100 20 100}]{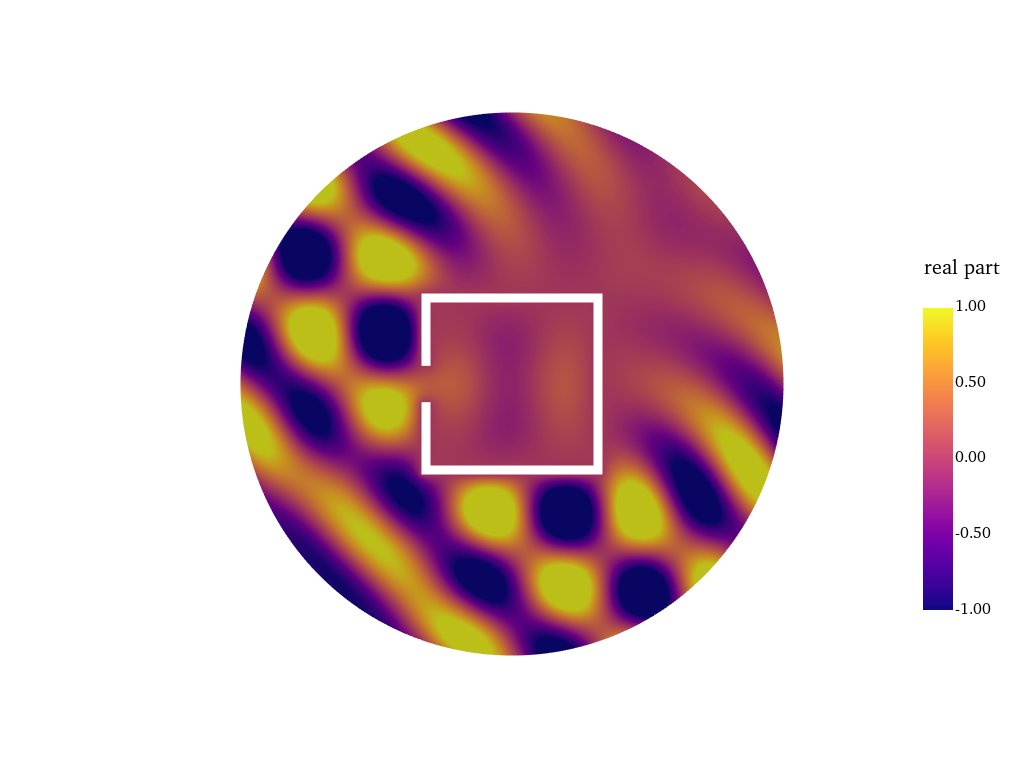}
    \caption{$k=10.6$.}
     \label{fig:trapping_square_total_nonres}
  \end{subfigure}
  \begin{subfigure}[t]{0.45\textwidth}
    \includegraphics[width=\textwidth,clip, trim={220 100 20 100}]{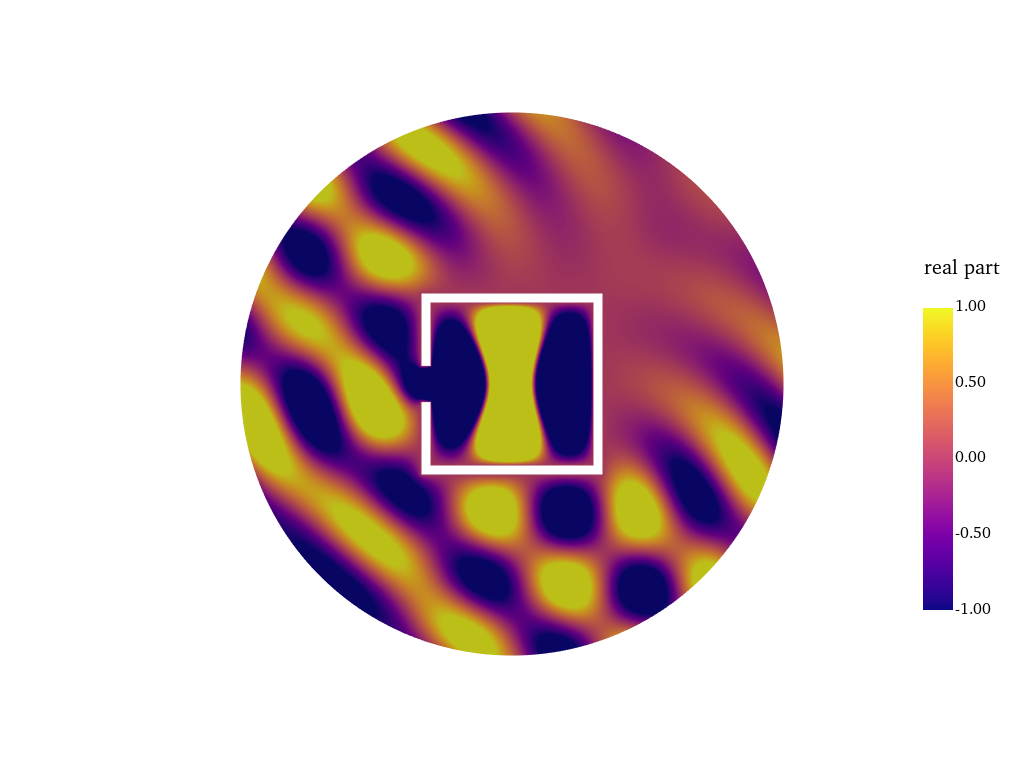}
    \caption{$k=10.8$.}
    \label{fig:trapping_square_total_resonant}
  \end{subfigure}
  \caption{Total fields for a non-resonant and a resonant wavenumber.}
  \label{fig:trapping_square_total}
\end{figure}

\subsection{Three-dimensional examples}
The same formulation applies in three spatial dimensions. We use two types of 3D experiments to test this extension. The first 3D test is a manufactured exterior Dirichlet benchmark based on a shifted outgoing point-source field. The source singularity is placed strictly inside the excluded obstacle, so the reference field solves the homogeneous Helmholtz equation throughout the computational exterior. This provides exact boundary data and an exact reference solution without introducing an incident plane wave. The second test uses the BeTSSi submarine geometry to assess the method on a realistic non-spherical obstacle.

\subsection{Manufactured outgoing-wave benchmark exterior to a ball}\label{sec:sphere}
To demonstrate high-order convergence in three dimensions, let
\begin{align*}
  c=(0.2,0.2,0.2), \qquad
  \mathcal O=B_{R_{\mathrm{scat}}}(c), \qquad
  R_{\mathrm{scat}}=1 .
\end{align*}
The compactification interface is \(R=1.5\), and mapped infinity is located at \(\rho=S\), with \(S=2>R\). We choose the source point $\tilde{x}_s=(0.6,0.6,0.6)$.
Because $\|\tilde{x}_s-c\|=\sqrt{0.48}<R_{\mathrm{scat}}$,
the singularity lies strictly inside the excluded obstacle. Hence
\[
  \Phi_k(\tilde{x})
  =
  \frac{\exp\!\left(ik\|\tilde{x}-\tilde{x}_s\|\right)}
       {\|\tilde{x}-\tilde{x}_s\|}
\]
is smooth in \(\mathbb R^3\setminus\overline{\mathcal O}\), satisfies
$(\Delta_{\tilde{x}}+k^2)\Phi_k=0$, and is outgoing. We prescribe $U=\Phi_k$
on $\partial\mathcal O$
and compare the numerical solution with \(\Phi_k\) in \(\mathcal D_{\mathrm{int}}\). The function \(\Phi_k\) is the outgoing 3D free-space Green function up to the conventional factor \(4\pi\). 
\begin{figure}[h!]
  \centering
  \begin{subfigure}[t]{0.45\textwidth}
    \includegraphics[width=\textwidth,clip, trim={220 100 20 100}]{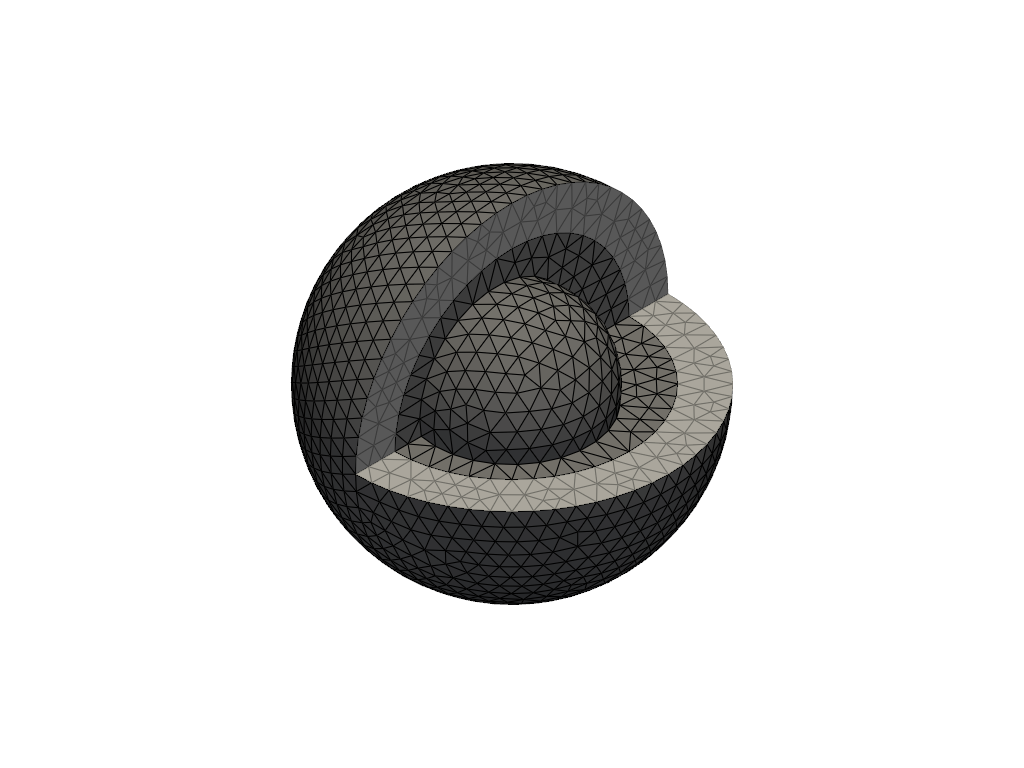}
	    \caption{Cutaway view of the curved tetrahedral mesh, including $\mathcal D_{\mathrm{int}}$ and the compactification layer.}
     \label{fig:sphere_mesh}
  \end{subfigure}
  \begin{subfigure}[t]{0.45\textwidth}
    \includegraphics[width=\textwidth,clip, trim={220 100 20 100}]{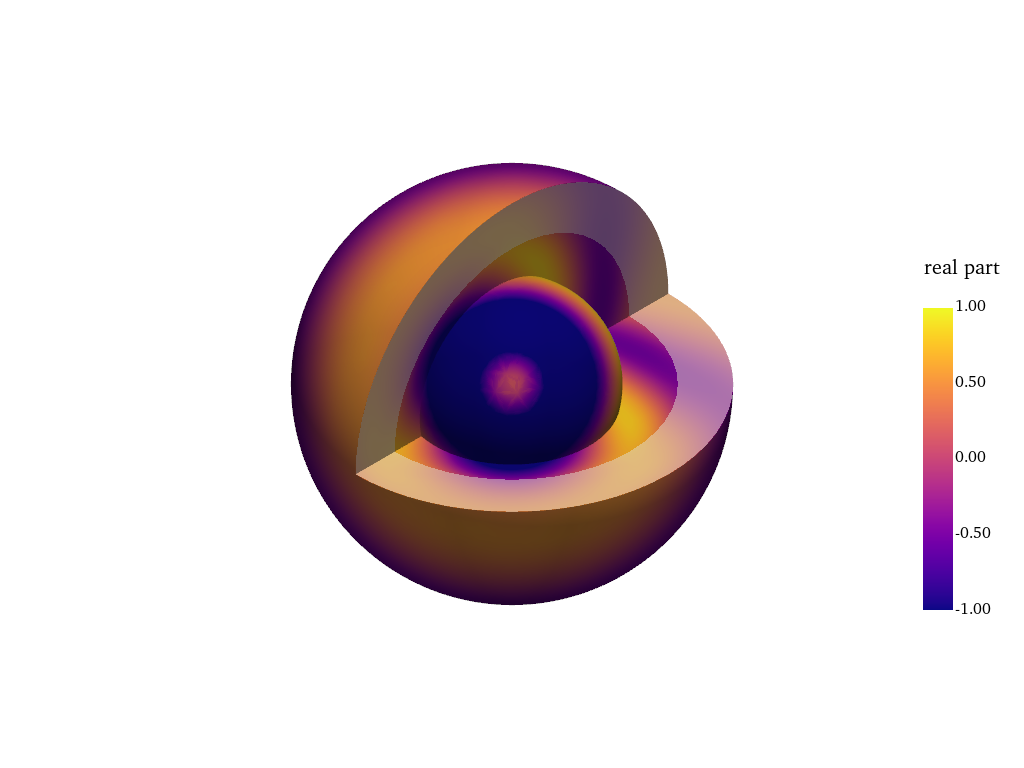}
	    \caption{Real part of the compactified numerical solution for $k=5$, $h=0.2$, and $p=2$.}
    \label{fig:sphere_sol}
  \end{subfigure}
	  \caption{Manufactured outgoing-wave benchmark exterior to an off-centered ball.}
  \label{fig:sphere_plots}
\end{figure}

Figure~\ref{fig:sphere_plots} shows a cutaway view of the mesh and the real part of the compactified numerical solution for \(h=0.2\), polynomial degree \(p=2\), and \(k=5\).

\begin{figure}[h!]
  \centering
  \begin{subfigure}[t]{0.45\textwidth}
    \includegraphics[width=\textwidth]{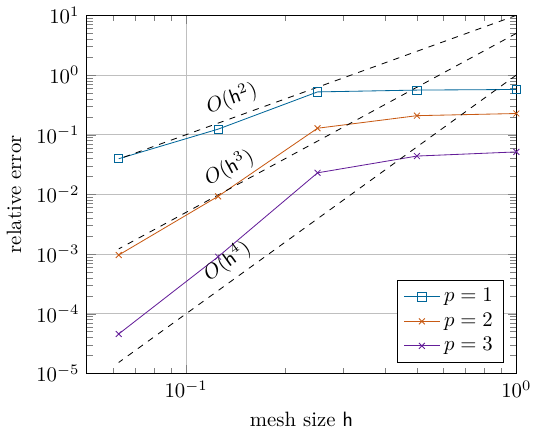}
    \caption{Relative $L^2$ error of the solution in $\mathcal D_{int}$ for various polynomial orders.}
     \label{fig:sphere_conv_orders}
  \end{subfigure}
  \begin{subfigure}[t]{0.45\textwidth}
    \includegraphics[width=\textwidth]{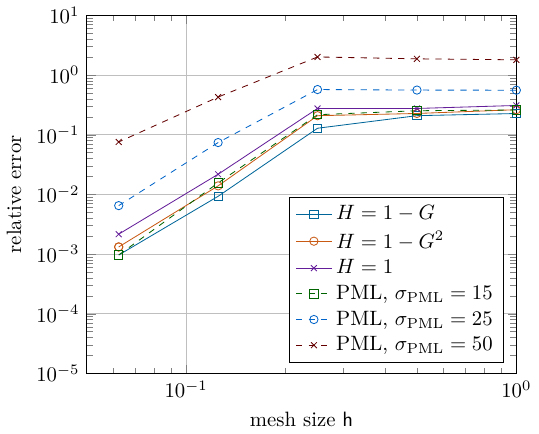}
    \caption{Relative $L^2(\mathcal D_{\mathrm{int}})$ error for the indicated boost functions and PML damping strengths at $p=2$.}
     \label{fig:sphere_conv_methods}
  \end{subfigure}
  \caption{Convergence for the shifted-Green-function manufactured benchmark at $k=5$, with $R=1.5$ and $S=2$.}
     \label{fig:sphere_conv}
\end{figure}

Figure~\ref{fig:sphere_conv} reports the relative \(L^2(\mathcal D_{\mathrm{int}})\) error with respect to \(\Phi_k\). Again, as expected, the observed slopes are consistent with \(\mathsf h^{p+1}\) convergence for the tested polynomial degrees. The influence of the choice of $H$ on the performance of the method is rather marginal; however, we again observe a drastic change of the error constant when varying the PML damping parameter.

Figure~\ref{fig:sphere_wavenumbers} shows the dependence of the relative $L^2(\mathcal D_{\mathrm{int}})$ error on the wavenumber for fixed mesh size $0.2$ and polynomial order $2$. The left panel uses geometrically distributed wavenumbers from $10^{-5}$ to $10^2$, while the right panel samples the transition region linearly with $k=5,5.5,\ldots,20$. The PML uses the real damping strength $\sigma_{\mathrm{PML}}=15$. As in the 2D disk experiment, the compactified formulation gives smaller errors in the low-wavenumber regime for the tested PML configuration, while the PML error becomes smaller once the fixed discretization enters the high-wavenumber, under-resolved regime. Compared with the disk experiment, this transition occurs at a smaller wavenumber, around $k\approx 5.8$, and the errors in the linear scan are much larger. This is consistent with the 3D benchmark using a coarser, lower-order discretization, so the interval around $k=10$ is already underresolved.

\begin{figure}[h!]
  \centering
  \begin{subfigure}[t]{0.45\textwidth}
    \includegraphics[width=\textwidth]{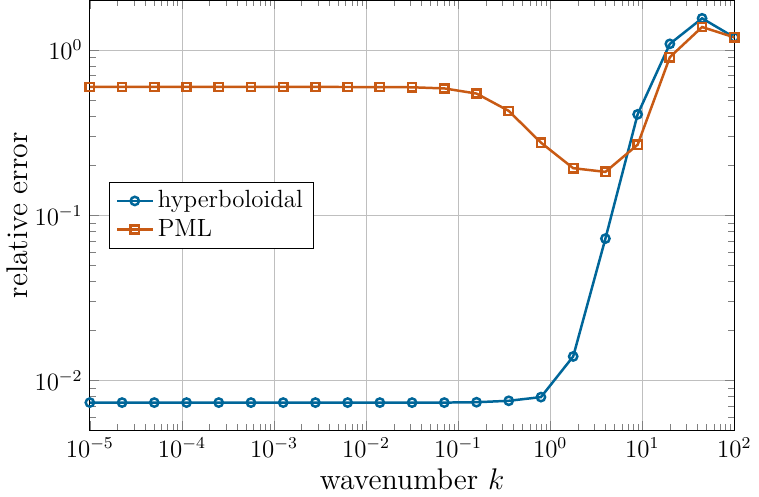}
    \caption{Geometrically distributed wavenumbers from $10^{-5}$ to $10^2$.}
    \label{fig:sphere_wavenumbers_log}
  \end{subfigure}
  \begin{subfigure}[t]{0.45\textwidth}
    \includegraphics[width=\textwidth]{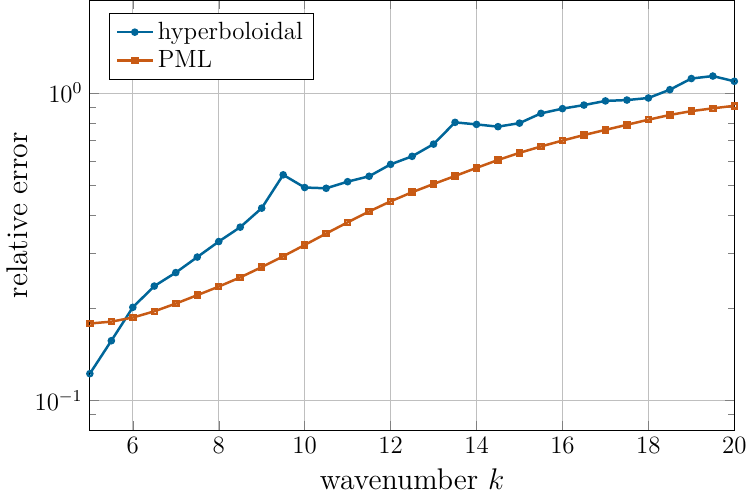}
    \caption{Linearly distributed wavenumbers in the transition region around $k=10$.}
    \label{fig:sphere_wavenumbers_zoom}
  \end{subfigure}
    \caption{Relative $L^2(\mathcal D_{\mathrm{int}})$ error for the shifted-Green-function manufactured benchmark, using $h=0.2$, $p=2$, and $\sigma_{\mathrm{PML}}=15$, as a function of the wavenumber $k$.}
    \label{fig:sphere_wavenumbers}
\end{figure}

\subsection{Scattering by a submarine}

To demonstrate the applicability of our method to more complex domains, we use the BeTSSi submarine test case
originally introduced as a Benchmark Target Strength Simulation test case for acoustic submarine
scattering \cite{schneider2003betssi}.
The CAD model used here is the BeTSSi geometry distributed in the NIRD Research Data
Archive \cite{venas2019betssi_data, venas2020igabem_submarine}.

The submarine is embedded in a ball of radius \(R=35\,\mathrm{m}\). We surround the computational domain by a layer with mapped infinity at \(S=40\,\mathrm{m}\). Figure~\ref{fig:submarine_meshes} shows the mesh with mesh size \(2\,\mathrm{m}\) of the submarine and the computational domain. Figure~\ref{fig:submarine_fields} shows the resulting approximations of the total and scattered fields, computed with second-order finite elements. The discretization has approximately \(3\times 10^5\) degrees of freedom.

\begin{figure}[h!]
  \centering
  \begin{subfigure}[t]{0.45\textwidth}
    \includegraphics[width=\textwidth,clip, trim={0 0 0 0}]{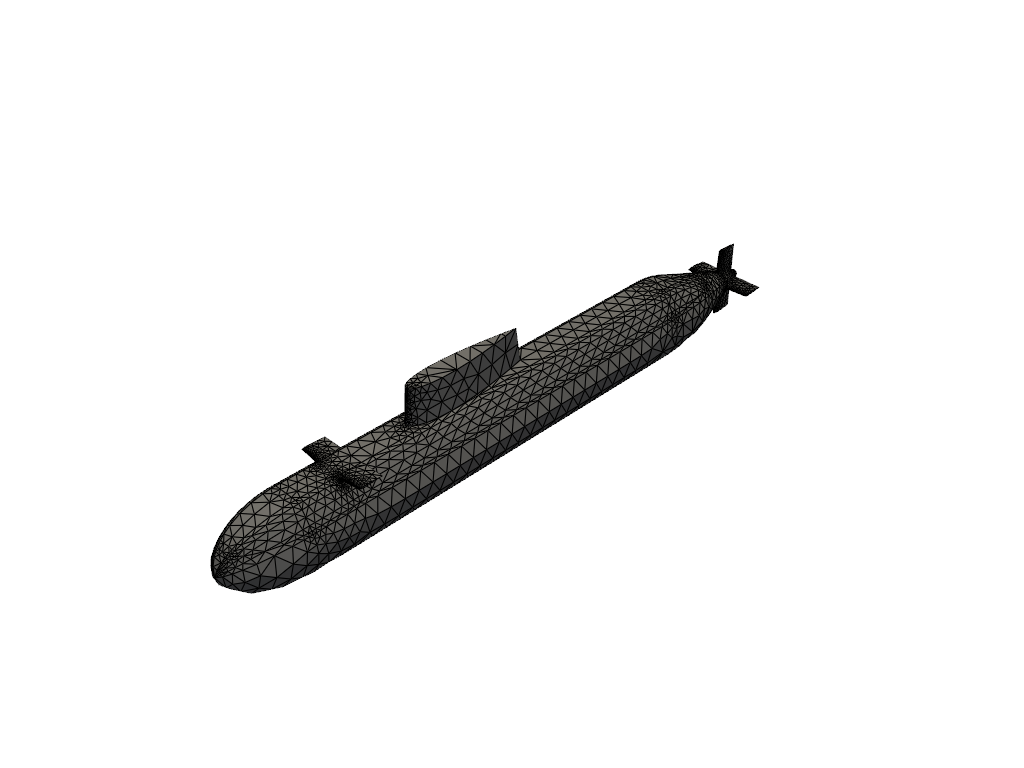}
    \caption{Mesh of the submarine.}
     \label{fig:submarine_only_sub_mesh}
  \end{subfigure}
  \begin{subfigure}[t]{0.45\textwidth}
    \includegraphics[width=\textwidth,clip, trim={0 0 0 0}]{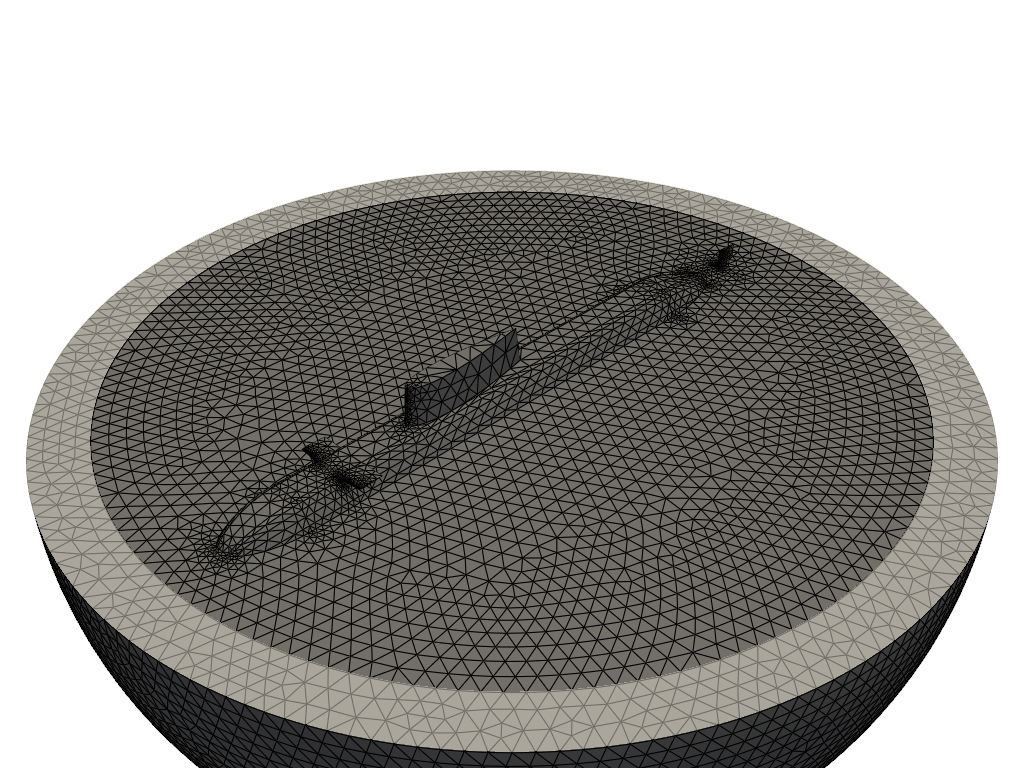}
    \caption{Mesh of the submarine and surrounding domains.}
    \label{fig:submarine_mesh}
  \end{subfigure}
  \caption{Meshes of the submarine experiment.}
  \label{fig:submarine_meshes}
\end{figure}

\begin{figure}[h!]
  \centering
  \begin{subfigure}[t]{0.45\textwidth}
    \includegraphics[width=\textwidth,clip, trim={220 100 20 100}]{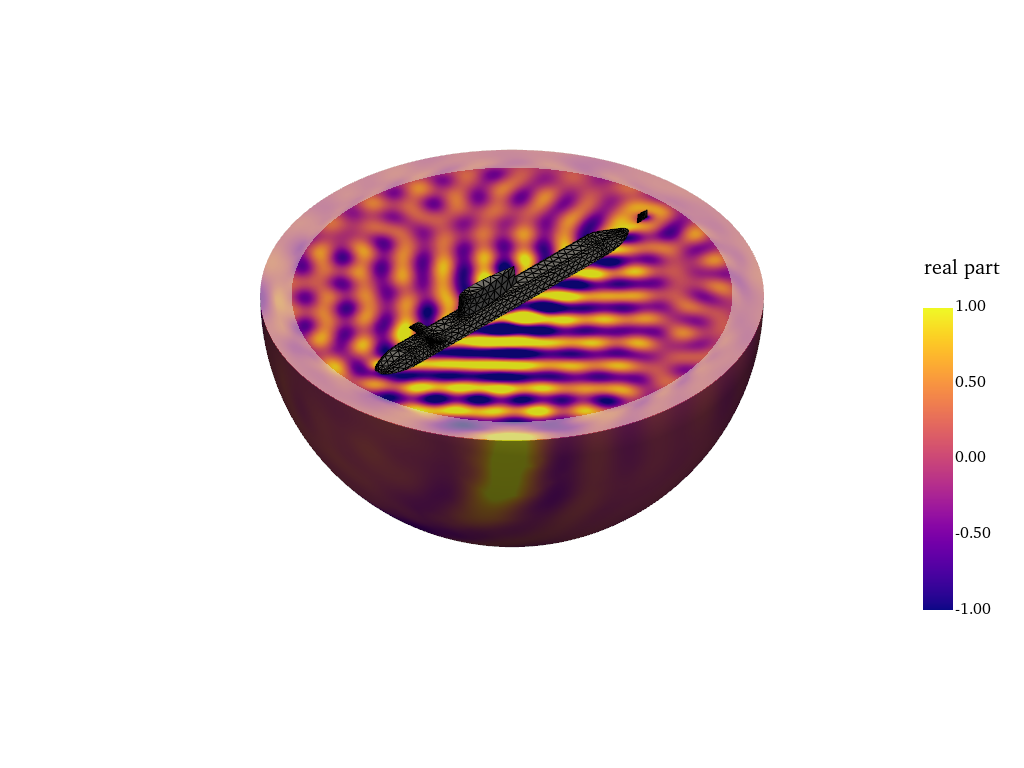}
    \caption{The real part of the scattered field, including the far field.}
     \label{fig:submarine_scattered}
  \end{subfigure}
  \begin{subfigure}[t]{0.45\textwidth}
    \includegraphics[width=\textwidth,clip, trim={220 100 20 100}]{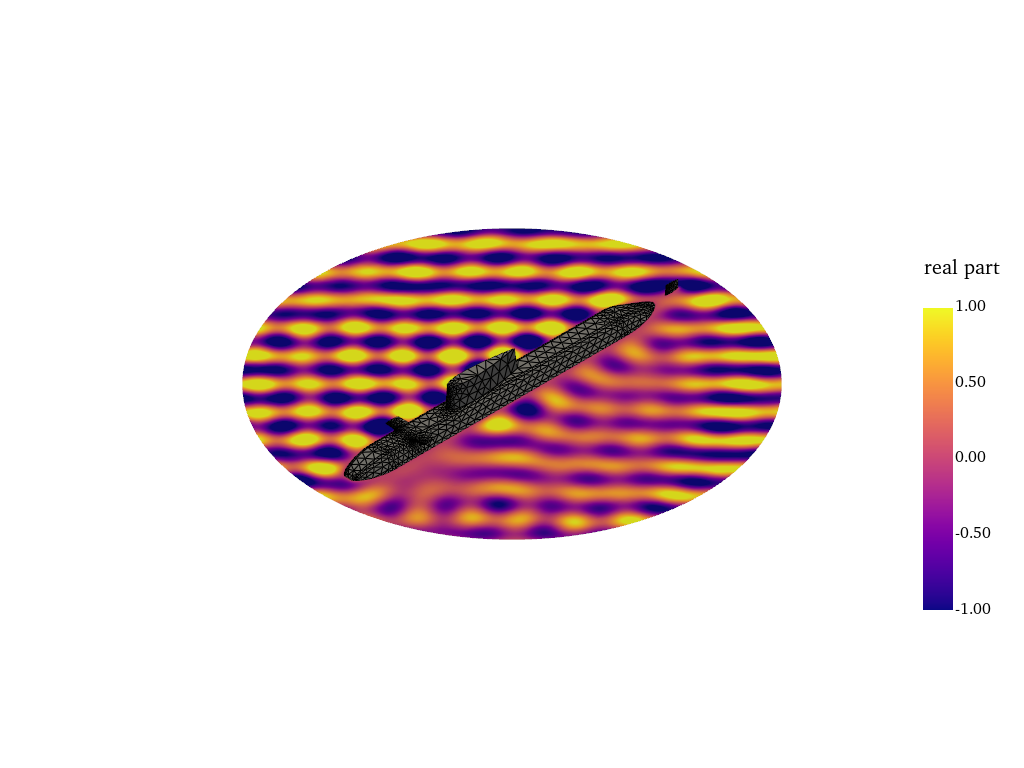}
    \caption{The real part of the total field.}
    \label{fig:submarine_total}
  \end{subfigure}
  \caption{Scattering of a plane wave (incoming from the top) on a submarine.}
  \label{fig:submarine_fields}
\end{figure}

\section{Conclusions}
\label{sec:conclusions}

We derived a finite element formulation of hyperboloidal compactification for the exterior Helmholtz equation. The construction combines a smooth radial compactification layer with a rescaling of the outgoing field. For smooth boundary defining functions and boost functions satisfying the boundedness condition on $(1-H^2)/G$, we proved that the pulled-back weak form has bounded coefficients on the compactified domain. The compactified boundary $\scri$ requires no essential boundary condition; the outgoing condition is encoded by the phase transformation and by the $H^1$ regularity of the compactified unknown. The trace at $\scri$ gives the far-field pattern up to an explicit normalization.

The numerical experiments demonstrate the practical behavior of this
formulation for high-order finite elements. For the disk benchmark, the
method shows the expected convergence both in the near field and in the
far-field pattern. The manufactured benchmark demonstrates the
corresponding high-order near-field convergence in three dimensions.
Comparisons with radial PML discretizations show comparable high-order
convergence in the resolved regime, while the compactified formulation
provides far-field data directly at $\scri$. For the PML configurations
tested here, the compactified formulation gives smaller errors in the
low-wavenumber regime. The trapping square experiment in 2D and the
submarine example in 3D show the applicability of the formulation to
more complex geometries.

For future work, the method should be extended to more general compactifications to handle non-spherical boundaries adapted to elongated obstacles. The analysis should be extended to a full stability and well-posedness theory for the resulting non-selfadjoint finite element problems. Extensions to heterogeneous media, more general boundary conditions, Maxwell fields, and large-scale three-dimensional benchmarks remain important open directions.



\section*{Acknowledgments}

The authors thank ESI for hospitality during the conference on ``Hyperboloidal Foliations and their Application.'' AZ was supported by the National Science Foundation under Grant No.~2309084. 

\section*{Declaration of generative AI and AI-assisted technologies in the manuscript preparation process}

During the preparation of this work, AZ used ChatGPT to assist with coding and manuscript review. After using this tool, the authors reviewed and edited the manuscript as needed and take full responsibility for the content of the published article.

\appendix

\section{Exact solution for scattering at a disk}
\label{sec:ref_sol}
\subsection{Centered}


Throughout this appendix, coordinates are physical coordinates. To avoid conflict with the compactified coordinate used in the main text, we write them as $\tilde x=(\tilde x_1,\tilde x_2)$.

We have a disk $D := \{\tilde x\in\mathbb{R}^2:\ |\tilde x|<1\}$, and we solve the Helmholtz equation for $k>0$ on the complement $D^c := \mathbb{R}^2\setminus \overline{D}$ with incident plane wave propagating in the $+\tilde x_1$ direction, $U^{\mathrm{i}}(\tilde x_1,\tilde x_2)=e^{ik\tilde x_1}$. We seek the scattered field $U^{\mathrm{s}}$ such that the total field
$U^{\mathrm{tot}}=U^{\mathrm{i}}+U^{\mathrm{s}}$ satisfies the exterior Dirichlet scattering problem
\begin{align}
(\Delta + k^2)\,U^{\mathrm{s}} &= 0 && \text{in }D^c,\\
U^{\mathrm{s}} &= -U^{\mathrm{i}} && \text{on }\partial D\ (r=1),\\
\lim_{r\to\infty}\sqrt{r}\,\left(\partial_r U^{\mathrm{s}} - i k U^{\mathrm{s}}\right) &= 0 && \text{(Sommerfeld)}.
\end{align}

In polar coordinates $(r,\theta)$ with $\tilde x_1=r\cos\theta$ and $\tilde x_2=r\sin\theta$, the Jacobi--Anger expansion gives
\[
U^{\mathrm{i}}(r,\theta)=e^{ikr\cos\theta}
=\sum_{n=-\infty}^{\infty} i^{\,n}\,J_n(kr)\,e^{in\theta},
\]
where $J_n$ is the Bessel function of the first kind. A radiating ansatz for the scattered field is
\[
U^{\mathrm{s}}(r,\theta)=\sum_{n=-\infty}^{\infty} a_n\,H_n^{(1)}(kr)\,e^{in\theta},
\qquad r>1,
\]
where $H_n^{(1)}$ is the outgoing Hankel function. Imposing the boundary
condition $U^{\mathrm{i}}+U^{\mathrm{s}}=0$ at $r=1$ yields
$i^{\,n}J_n(k)+a_nH_n^{(1)}(k)=0$, hence
$a_n=-i^{\,n}J_n(k)/H_n^{(1)}(k)$ for each Fourier mode $n$.
Therefore the explicit scattered field is
\[
U^{\mathrm{s}}(r,\theta)
=
-\sum_{n=-\infty}^{\infty}
i^{\,n}\,\frac{J_n(k)}{H_n^{(1)}(k)}\,
H_n^{(1)}(kr)\,e^{in\theta},
\qquad r>1.
\]

We define the far-field pattern at infinity, $u_\infty(\theta)$, by
$U^{\mathrm{s}}(r,\theta)\sim e^{ikr}u_\infty(\theta)/\sqrt r$ as
$r\to\infty$.
Using the large-argument asymptotic for the radiating Hankel function
\[
H_n^{(1)}(kr)
\sim
\sqrt{\frac{2}{\pi kr}}\,
\exp\!\left(i\left(kr-\tfrac{n\pi}{2}-\tfrac{\pi}{4}\right)\right),
\qquad r\to\infty,
\]
we obtain the 2D far-field
\begin{equation}\label{eq:disk_farfield}
u_\infty(\theta)
=
-\sqrt{\frac{2}{\pi k}}\,e^{-i\pi/4}
\sum_{n=-\infty}^{\infty}
\frac{J_n(k)}{H_n^{(1)}(k)} e^{in\theta}.
\end{equation}

\subsection{Off-centered}

An off-centered unit disk is given by $D_a := \{\tilde x\in\mathbb{R}^2:\ |\tilde x-a|<1\}$, where the offset is $a=(a_1,a_2)\in\mathbb{R}^2$. The exterior Dirichlet scattering problem is the same as in the centered case, except that the boundary condition is applied on the boundary of the off-centered disk. To write the solution explicitly, we introduce polar coordinates about the offset $a$ by
$\tilde x=a+r_a(\cos\phi,\sin\phi)$, with $r_a:=|\tilde x-a|$ and
$\phi:=\arg(\tilde x-a)$.
Then, since $\tilde x_1=a_1+r_a\cos\phi$,
\[
U^{\mathrm{i}}(\tilde x)=e^{ik\tilde x_1}=e^{ik(a_1+r_a\cos\phi)}=e^{ika_1}\,e^{ikr_a\cos\phi}=e^{ika_1}\sum_{n=-\infty}^{\infty} i^{\,n}J_n(kr_a)\,e^{in\phi},
\]
where we used the Jacobi--Anger expansion in the last step.
A radiating separated-variable ansatz for the scattered field about $a$ is
\[
U^{\mathrm{s}}(r_a,\phi)=\sum_{n=-\infty}^{\infty} a_n\,H_n^{(1)}(kr_a)\,e^{in\phi},
\qquad r_a>1,
\]
where $H_n^{(1)}$ is the outgoing Hankel function. Imposing the Dirichlet boundary condition on the off-centered disk mode by mode, we arrive at the solution
\[
U^{\mathrm{s}}(\tilde x)
=
-\,e^{ika_1}\sum_{n=-\infty}^{\infty}
i^{\,n}\,\frac{J_n(k)}{H_n^{(1)}(k)}\,
H_n^{(1)}\!\left(k|\tilde x-a|\right)\,e^{in\arg(\tilde x-a)},
  \qquad |\tilde x-a|>1.
\]
To evaluate this solution at a centered radius about the origin, $|\tilde x|=R$, we note
$r_a(\theta)=|\tilde x-a|=\sqrt{(R\cos\theta-a_1)^2+
(R\sin\theta-a_2)^2}
=\sqrt{R^2+|a|^2-2R(a_1\cos\theta+a_2\sin\theta)}$ and
$\phi(\theta)=\arg(\tilde x-a)=\operatorname{atan2}(R\sin\theta-a_2,R\cos\theta-a_1)$.
Therefore,
\[
U^{\mathrm{s}}(R,\theta)
=
-\,e^{ika_1}\sum_{n=-\infty}^{\infty}
i^{\,n}\,\frac{J_n(k)}{H_n^{(1)}(k)}\,
H_n^{(1)}\!\left(kr_a(\theta)\right)\,e^{in\phi(\theta)}.
\]

For the far field, we note the observation direction $\hat x=(\cos\theta,\sin\theta)$ and the incident direction $d=(1,0)$.
Then we apply a translation rule to obtain the far-field for the shifted disk:
$u_\infty^{(a)}(\theta)=e^{ik(d-\hat x)\cdot a}u_\infty^{(0)}(\theta)
=e^{ik(a_1-\hat x\cdot a)}u_\infty^{(0)}(\theta)
=\exp(ik(a_1-a_1\cos\theta-a_2\sin\theta))u_\infty^{(0)}(\theta)$,
where $u_\infty^{(0)}(\theta)$ is given in \eqref{eq:disk_farfield}. To summarize, the off-centered far-field differs from the centered far-field only by a phase factor:
\[
u_\infty^{(a)}(\theta)
=
\exp\!\left(ik(a_1-a_1\cos\theta-a_2\sin\theta)\right)\,
\sqrt{\frac{2}{\pi k}}\,e^{-i\pi/4}
\sum_{n=-\infty}^{\infty}\left(-\frac{J_n(k)}{H_n^{(1)}(k)}\right)e^{in\theta}.
\]

\bibliographystyle{unsrt}
\bibliography{refs.bib}

\end{document}